\documentclass[11pt,reqno]{amsart}

\usepackage[utf8]{inputenc}
\usepackage{amsfonts}
\usepackage{amsmath}
\usepackage{mathrsfs}
\usepackage{amsthm}
\usepackage{amssymb}
\usepackage[top=30truemm,bottom=30truemm,left=30truemm,right=30truemm]{geometry}
\usepackage[bookmarks=false,draft=false,breaklinks,colorlinks]{hyperref}
\usepackage[dvipdfmx]{graphicx}
\usepackage[all]{xy}

\usepackage{color}
\usepackage{url}

\usepackage{comment}

\usepackage{fancyhdr}
\usepackage{enumitem}
\setenumerate{label=(\arabic*),nosep}
\setitemize{nosep}
\newlist{clist}{enumerate}{1}
\setlist*[clist]{label=(\roman*),nosep}

\makeatletter
\let\@fnsymbol\@arabic
\makeatother

\theoremstyle{definition}
\newtheorem{Def}{Definition}[section]
\newtheorem{Rem}[Def]{Remark}

\newtheorem{Que}[Def]{Question}

\theoremstyle{plain}
\newtheorem{Thm}[Def]{Theorem}

\newtheorem{Prop}[Def]{Proposition}
\newtheorem{Lem}[Def]{Lemma}
\newtheorem{Cor}[Def]{Corollary}
\newtheorem*{thm}{Theorem}

\newcommand{\quotient}[2]{
\mathchoice{  \text{\raise1ex\hbox{$#1$}\!\Big/\!\lower1ex\hbox{$#2$}} }
                  {  \text{\raise1pt\hbox{$#1$}\big/\lower1pt\hbox{$#2$}} }
                  {  {#1}\,/\,{#2}  }
                  {  {#1}\,/\,{#2}  }
}

\title{$B$-valued semi-circular system and the free Poincar\'{e} inequality}
\author{Hyuga Ito}
 
\address{
Graduate School of Mathematics, Nagoya University, Furocho, Chikusaku, Nagoya, 464-8602, Japan
}
\email{hyuga.ito.e6@math.nagoya-u.ac.jp}

\date{\today}

\begin{document}

\begin{abstract}
    We characterize $B$-valued semi-circular systems in terms of a $B$-valued free probabilistic analogue of Poincar\'{e} inequality. This is a $B$-valued generalization of Biane's theorem \cite[Theorem 5.1]{b03}. Moreover, we prove that Voiculescu's conjecture on $B$-valued free Poincar\'{e} inequality in \cite{aim06} is not in the affirmative as it is.
\end{abstract}
\maketitle

\allowdisplaybreaks{
\section{Introduction}
Free probability theory is a non-commutative probability theory based on the notion of free independence, initiated by Voiculescu \cite{v85} in the early 80s. In his work, Voiculescu proved a free probabilistic analogue of central limit theorem (see \cite[Theoem 4.8]{v85}). In the theorem, the \textit{semi-circular distribution}
$
d\sigma_{S}(t)
=
\frac{1}{2\pi}
1_{[-2,2]}(t)
\sqrt{4-t^2}\,dt
$ (with mean $0$ and variance $1$)
appears as the limit distribution, where $dt$ is the Lebesgue measure and $1_{[-2,2]}$ is the indicator function on $[-2,2]$. This implies that the semi-circular distribution plays a role of Gaussian distribution in free probability theory. Then, the semi-circular distribution has been studied by many hands. For example, Voiculescu proved a free probabilistic analogue of Stein's equation (see \cite[Proposition 3.8]{v98}) and Biane characterized the semi-circular distribution by means of (sharp) $\mathbb{C}$-valued free Poincar\'{e} inequality (see \cite[Theorem 5.1]{b03}). In his unpublished note, Voiculescu proved the $\mathbb{C}$-valued free Poincar\'{e} inequality 
\[
\|P(Y_{1},\dots,Y_{d})-\tau\left(P(Y_{1},\dots,Y_{d})\right)\|_{\tau}
\leq \sqrt{2}\max_{j\in[d]}\{
|\|Y_{j}\|\}\cdot\sum_{j=1}^{d}\|\partial_{j}P(Y_{1},\dots,Y_{d})\|_{\tau\otimes\tau},
\]
for any $P\in\mathbb{C}\langle X_{1},\dots,X_{d}\rangle$ and any non-commutative tuples $Y_{1},\dots,Y_{d}$ of self-adjoint elements in a von Neumann algebra $M$ with a faithful normal tracial state $\tau$, where $\mathbb{C}\langle X_{1},\dots,X_{d}\rangle$ is the unital $*$-algebra of all $\mathbb{C}$-valued non-commutative polynomials in $X_{j}^*=X_{j}$, $j\in\{1,2,\dots,n\}$ (see e.g. \cite[section 8.1, Theorem 5]{misp17}). In the above inequality, a non-commutative replacement $\partial_{j}$ of differential operator called the \textit{free difference quotient} appears, and it was introduced in the study of (non-microstate) free entropy and free Fisher's information measure (see \cite[section 2]{v98}). Voiculescu also proposed a conjecture on a $B$-valued analogue of free Poincar\'{e} inequality in \cite{aim06}, which means that $\mathbb{C}$-valued non-commutative polynomials are replaced with $B$-valued non-commutative polynomials. Related to this conjecture, the author \cite{i24} proved a ``weak'' free Poincar\'{e} inequality, in which the $L^2$-norm $\|\partial_{X:B}[P(X)]\|_{\tau\otimes\tau}$ was replaced with the projective tensor norm $\|\partial_{X:B}[P(X)]\|_{\widehat{\otimes}}$, to improve Voiculescu's lemma \cite[Lemma 3.4]{v00} about the kernel of the closure of free difference quotient. It is still an open problem whether or not $\mathrm{ker}(\overline{\partial}_{X:B})=L^2(B,\tau)$, where $\overline{\partial}_{X:B}$ denotes the $L^2$-closure of $\partial_{X:B}:L^2(B\langle X\rangle,\tau)\to L^2(B\langle X\rangle\otimes B\langle X\rangle,\tau\otimes\tau)$ (see \cite{aim06}).

In the mid 80s, Voiculescu \cite{v95} introduced the notion of \textit{$B$-free independence}, which arises by the construction of free product with amalgamation over $B$, where $B$ is a unital algebra over $\mathbb{C}$. In his work, Voiculescu proved a $B$-valued free probabilistic analogue of central limit theorem and called the limit distribution a \textit{$B$-valued semi-circular distribution} (see \cite[Definition 8.1 and Theorem 8.4]{v95}). Then, Speicher \cite{sp98} reconstructed the framework of $B$-valued free probability theory in terms of partitions and also defined the notion of $B$-valued semi-circular elements in terms of non-crossing partitions (see Definition \ref{Def_Bsemicircular}).

In this paper, we will generalize Biane's theorem \cite[Theorem 5.1]{b03} to the $B$-valued semi-circular distribution, that is, we will characterize the $B$-valued semi-circular system in terms of an ``appropriate'' $B$-valued free Poincar\'{e} inequality. The main result of this paper is the following:

\begin{thm}
    Let $(S_{1},\dots,S_{d})$ be a $B$-free family of self-adjoint non-commutative random variables in a tracial $B$-valued non-commutative probability space $(A,B,\tau,E)$ (see Definition \ref{Def_nonprobspace} below) with mean $0$ and variance $\eta=(\eta_{1},\dots,\eta_{d})$, respectively. Then, the following are equivalent:
    \begin{enumerate}
        \item $(S_{1},\dots,S_{d})$ is a $B$-valued semi-circular system associated with $\eta_{1},\dots,\eta_{d}$ (see Definition \ref{Def_Bsemicircular} below).
        \item We have
        \[
        \left\|
        P(S_{1},\dots,S_{d})
        -
        E\left[
        P(S_{1},\dots,S_{d})
        \right]
        \right\|_{\tau}^2
        \leq
        \sum_{j=1}^{d}
        \left\|
        \mathrm{ev}_{S}^{\otimes2}
        \left(
        \partial_{j}
        \left[
        P(X_{1},\dots,X_{d})
        \right]
        \right)
        \right\|_{\eta_{j}}^2
        \]
        for any non-commutative polynomial $P(X_{1},\dots,X_{d})$ of a certain class $\mathcal{T}_{\eta}B_{\langle d\rangle}$, which consists of all elements such that they admit expressions of ``particular'' linear combinations of finite products of $B$-valued Chebyshev polynomials (see section \ref{section_characterization} for the precise definition).
        \item We have
        \[
        \|P(S_{1},\dots,S_{d})-E[P(S_{1},\dots,S_{d})]\|_{\tau}^2\leq\sum_{j=1}^{d}\|\mathrm{ev}_{S}^{\otimes2}(\partial_{j}[P(X_{1},\dots,X_{d})])\|_{\eta_{j}}^2
        \]
        for any $P\in B_{\langle d\rangle}$.
    \end{enumerate}
\end{thm}
This is analogous to a classical fact about Poincar\'{e} inequalities for Gaussian random variables (see \cite[Corollary 1]{bu83}). To prove the main result, we will introduce two objects: One is a $B$-valued analogue of Chebyshev polynomials of the second kind, say $\{U^{\eta}_{n}\}_{n=1}^{\infty}$, which will be called the \textit{$B$-valued Chebyshev family} (see Definition \ref{Def_B_chebyshev}). We believe that the family together with its properties is of independent interest beyond the main result of this paper. Its definition is similar to the classical one, but depends on a given linear map $\eta:B\to B$. Also, the $B$-valued Chebyshev family has a remarkable matrix amplification property, which allows us to describe every element in $B_{\langle d\rangle}$ as an element in $\mathcal{T}_{\eta\otimes \mathrm{id}_{N}}\left(M_{N}(B)\right)_{\langle d\rangle}$ with a sufficiently large $N\in\mathbb{N}$ (see Lemma \ref{lem_matampChebyshev2} and Proposition \ref{prop_chebyshevorthog_formla}). This lemma will be used in the proof of ((2)$\Leftrightarrow$)(1)$\Rightarrow$(3) of the main result. The other is the \textit{divergence operator} $\partial_{j}^*$ with respect to a $B$-valued semi-circular element, which is a $B$-valued (algebraic) analogue of the adjoint $\partial_{X:B}^*$ of $\partial_{X:B}$, which appeared in \cite[Proposition 4.3]{v98} (see Definition \ref{Def_divergence}). 
Similarly to the (classical) Chebyshev polynomials of the second kind, the $B$-valued Chebyshev family $\{U^{\eta}_{n}\}_{n=1}^{\infty}$ are orthogonal with respect to the $B$-valued semi-circular distribution (see Propositions \ref{Prop_Bsemicircular_equivalent} and \ref{Prop_BChebyshev_tensor} and Corollary \ref{cor_orthogonality_chebyshev}) and each $U_{n}^{\eta}$ provides an eigenvector of $\partial_{j}^*\circ\partial_{j}$ with the eigenvalue $n$ (see Proposition \ref{Prop_numberoperator}). Remark that the $L^2$-norm $\|\cdot\|_{\eta_{j}}$ in the above is different from that in Voiculescu's conjecture and also that $\mathcal{T}_{\eta}B_{\langle d\rangle}$ is, in general, not equal to the all non-commutative polynomials $B_{\langle d\rangle}=B\langle X_{1},\dots, X_{d}\rangle$ (see Remark \ref{Rem_exmaple}).
As a corollary of the main result, it will be confirmed that the kernel of the free difference quotient associated with variance with respect to $B$-valued semi-circular system is exactly $L^2(B,\tau)$ (see Corollary \ref{cor_kernel}). This is a variant of Voiculescu's conjecture proposed in \cite{aim06}.

Moreover, we will give a certain simple counterexample to Voiculescu's conjecture on $B$-valued free Poincar\'{e} inequality ``in a naive sense''. 

This paper consists of 6 sections from \ref{section_notation} to \ref{section_counterexample} besides this introduction. 

In section \ref{section_notation}, we will prepare notations, definitions and facts, which will be used later.

In section \ref{section_B_Chebyshev}, we will introduce a $B$-valued generalization of Chebyshev polynomials of the second kind and study its basic properties.

In section \ref{section_conjugate}, we will prove a $B$-valued analogue of Stein's equation for the $B$-free semi-circular system.

In section \ref{section_divergence}, we will introduce the divergence operator in the $B$-valued framework and prove that the composition of the free difference quotient and the divergence operator plays a role of number operator.

In section \ref{section_characterization}, we will prove our main result.

In section \ref{section_counterexample}, we will give a counterexample to Voiculescu's conjecture on $B$-valued free Poincar\'{e} inequality.

Some facts proven in sections \ref{section_conjugate} and \ref{section_divergence} are also seen in \cite{sh00}. However, since our setting is different from his, we will give precise arguments (by different method from his).

\section{Notations and preliminaries}\label{section_notation}
For any $k\in\mathbb{N}$, we write $[k]:=\{1,\dots,k\}$. For any $k,d\in\mathbb{N}$, we denote by $I(k,d)$ the set of all maps from $[k]$ to $[d]$. In particular, we set 
\[
\mathrm{Alt}(I(k,d))=\left\{i\in I(k,d)\,\middle|\,i(j)\not=i(j+1)\mbox{ for all }j\in[k-1]\right\}.
\]
Moreover, we denote by $I(k,\mathbb{N})$ the set of all maps from $[k]$ to $\mathbb{N}$. Let $\mathbb{C}_{\langle d\rangle}=\mathbb{C}\langle X_{1},\dots,X_{d}\rangle$ be the unital algebra of all $\mathbb{C}$-coefficients non-commutative polynomials in $X=(X_{1},\dots,X_{d})$. Let $B$ be a unital algebra over $\mathbb{C}$. We denote by $B_{\langle d\rangle}=B\langle X_{1},\dots,X_{d}\rangle$ the (algebraic) free product of $B$ and $\mathbb{C}_{\langle d\rangle}$.
If $B$ is a unital $*$-algebra and $X_{j}=X_{j}^*$, $j\in[d]$, then $B_{\langle d\rangle}$ naturally becomes a unital $*$-algebra over $\mathbb{C}$. 

Let $A$ be a unital algebra over $\mathbb{C}$ such that $B$ is a unital subalgebra of $A$, and $Y=(Y_{1},\dots,Y_{d})$ elements of $A$. Then, by the universality of free product, there exists a unique algebra homomorphism $\mathrm{ev}_{Y}$ from $B_{\langle d\rangle}$ to $A$ such that $\mathrm{ev}_{Y}(X_{j})=Y_{j}$, $j\in[d]$, and $\mathrm{ev}_{Y}(b)=b$ for all $b\in B$. If $B\subset A$ is a unital inclusion of unital $*$-algebras and $Y_{j}=Y_{j}^*$, $j\in[d]$, then $\mathrm{ev}_{Y}$ becomes a $*$-homomorphism. For any $P(X)=P(X_{1},\dots,X_{d})\in B_{\langle d\rangle}$, we often write $\mathrm{ev}_{Y}(P(X))=P(Y)$.

When $d=1$, we write $X=X_{1}$ for the simplicity of our notation. Then,  we denote by $\partial=\partial_{X:B}$ the \textit{free difference quotient} with respect to a formal random variable $X$ over $B$, which was introduced in \cite{v98} and is a unique derivation from $B\langle X\rangle$ to $B\langle X\rangle\otimes B\langle X\rangle$ such that
\[
\partial[X]=1\otimes1\mbox{ and }B\subset\ker(\partial).
\]
It is easy to see that
\[
\partial
\left[
b_{0}Xb_{1}\cdots Xb_{k}
\right]
=
\sum_{1\leq \ell\leq k}
b_{0}Xb_{1}\cdots Xb_{\ell-1}\otimes b_{\ell}Xb_{\ell+1}\cdots Xb_{k}
\]
for any $b_{0},b_{1},\dots,b_{k}\in B$. When $X=(X_{1},\dots,X_{d})$, $d\geq2$, identifying
\[
B_{\langle d\rangle}=\left(B\langle X_{1},\dots,X_{j-1},X_{j+1},\dots,X_{d}\rangle\right)\langle X_{j}\rangle,
\]
we set $\partial_{j}=\partial_{X_{j}:B\langle X_{1},\dots,X_{j-1},X_{j+1},\dots,X_{d}\rangle}$ for each $j\in[d]$.

\begin{Def}\label{Def_nonprobspace}
    Let $B\subset A$ be a unital inclusion of unital $*$-algebras such that $B$ is assumed to be a unital $C^*$-algebra. A \textit{(positive) conditional expectation} $E$ (onto $B$) is a linear map $E:A\to B$ such that $E[b_{1}ab_{2}]=b_{1}E[a]b_{2}$ for all $a\in A$ and $b_{1},b_{2}\in B$ and that $E[a^*a]\geq0$ for all $a\in A$. The triple $(A,B,E)$ is called a \textit{$B$-valued (non-commutative) probability space}. An element of $A$ is called a \textit{$B$-valued (non-commutative) random variable}. In particular, when there is a faithful tracial state $\tau$ on $A$ such that $\tau\circ E=\tau$, we call the tuple $(A,B,\tau,E)$ a \textit{tracial $B$-valued (non-commutative) probability space}.
\end{Def}

\begin{Def}
    Let $(A,B,E)$ be a $B$-valued probability space and $a_{1},\dots,a_{d}$ be $B$-valued random variables. We say that $a_{1},\dots,a_{n}$ are \textit{$B$-freely independent} if 
    \[
    E\left[
    \left(P_{1}(a_{i(1)})-E\left[P_{1}(a_{i(1)})\right]\right)\cdots \left(P_{k}(a_{i(k)})-E\left[P_{k}(a_{i(k)})\right]\right)
    \right]
    =0
    \]
    whenever $k\in\mathbb{N}$, $P_{1},\dots,P_{k}\in B\langle X\rangle$ and $i(\cdot)\in\mathrm{Alt}(I(k,d))$.
\end{Def}

\begin{Def}\label{Def_Bsemicircular}
    Let $(A,B,E)$ be a $B$-valued probability space and $\eta$ a completely positive map from $B$ to $B$. A $B$-valued random variable $S=S^*$ is called a \textit{$B$-valued semi-circular element} with mean $0$ and variance $\eta$ if 
    \[
    E[b_{0}Sb_{1}\cdots Sb_{k}]
    =
    \sum_{\pi\in NC_{2}(k)}
    b_{0}
    \eta_{\pi}
    (b_{1},\dots,b_{k-1})
    b_{k}
    \]
    for any $k\in\mathbb{N}$ and any $b_{0},b_{1},\dots,b_{k}$, where $\eta_{\pi}$ is the multiplicative $B$-valued functional determined from $\eta_{1_{2}}(b)=\eta(b)=E[SbS]$ for any $b\in B$; for example, $\eta_{\{\{1,4\},\{2,3\}\}}(b_{1},b_{2},b_{3})=\eta(b_{1}\eta(b_{2})b_{3})$. (See \cite{sp98} for details on multiplicative $B$-valued maps. Also, an explanation on intuitive treatments for multiplicative $B$-valued maps is given in \cite{abfn12}.) If $S_{1},\dots,S_{d}$ are $B$-freely independent in $(A,B,E)$ and $S_{j}$ is a $B$-valued semi-circular element with mean $0$ and variance $\eta_{j}$, which is a completely positive map from $B$ to $B$, then $S_{1},\dots,S_{d}$ are called a \textit{$B$-free $B$-valued semi-circular system associated with $\eta_{1},\dots,\eta_{d}$}.
\end{Def}

We have a realization of $B$-free $B$-valued semi-circular system associated with $\eta_{1},\dots,\eta_{d}$ as follows. Let $\eta_{1},\dots,\eta_{d}$ be completely positive maps from $B$ to $B$. Let $\mathcal{F}$ be the $B$-valued (algebraic) full Fock space, that is, 
\[
\mathcal{F}
=
B\oplus
\bigoplus^{\mathrm{alg}}_{m\geq1}\bigoplus^{\mathrm{alg}}_{i\in I(m,d)}
BX_{i(1)}BX_{i(2)}\cdots BX_{i(m)}B
\]
with a $B$-valued pre-inner product $\langle\cdot,\cdot\rangle_{\mathcal{F}}$ defined by
\begin{align*}
    \,&
    \left\langle
    b_{0}X_{i(1)}b_{1}\cdots X_{i(k)}b_{k},
    b'_{0}X_{j(1)}b_{1}'\cdots X_{j(\ell)}b_{\ell}'
    \right\rangle_{\mathcal{F}}\\
    &=
    \delta_{k,\ell}\delta_{i(1),j(1)}\cdots\delta_{i(k),j(k)}
    b_{k}^*\eta_{i(k)}(b_{k-1}^*\eta_{i(k-1)}(b_{k-2}^*\cdots\eta_{i(2)}(b_{1}^*\eta_{i(1)}(b_{0}^*b_{0}')b_{1}')\cdots b_{k-2}')b_{k-1}')b_{k}'
\end{align*}
and
\begin{align*}
    \langle b+P,b'+P'\rangle_{\mathcal{F}}=b^*b'+\langle P,P'\rangle_{\mathcal{F}}
\end{align*}
for any $k,\ell\in\mathbb{N}$, any $b, b_{0},b_{1},\dots,b_{k},b',b_{0}',b_{1}',\dots,b_{\ell}'\in B$ and $P,P'\in \mathcal{F}\ominus B$.

Let $\ell_{j}$ be the linear operator on $\mathcal{F}$ defined by
\[
\ell_{j}\left[b_{0}X_{i(1)}b_{1}\cdots X_{i(k)}b_{k}\right]
=X_{j}b_{0}X_{i(1)}b_{1}\cdots X_{i(k)}b_{k}
\]
for any $k\in\{0\}\cup\mathbb{N}$, any $i\in I(k,d)$ and any $b_{0},b_{1},\dots,b_{k}\in B$. Namely, $\ell_{j}$ is the creation operator with respect to $X_{j}$. Also, let $\ell_{j}^*$ be the linear operator such that 
\[
\ell_{j}^*[b]=0\mbox{ and }
\ell_{j}^*\left[
b_{0}X_{i(1)}b_{1}\cdots X_{i(k)}b_{k}
\right]
=
\delta_{j,i(1)}\eta_{j}(b_{0})b_{1}X_{i(2)}b_{2}\cdots X_{i(k)}b_{k}\,(k\geq1)
\]
for any $b,b_{0},b_{1},\dots,b_{k}\in B$. Then, $\ell_{j}^*$ is actually the adjoint of $\ell_{j}$ with respect to $\langle\cdot,\cdot\rangle_{\mathcal{F}}$, that is, 
\[
\left\langle
\ell_{j}[\xi],\xi'
\right\rangle_{\mathcal{F}}
=
\left\langle
\xi,\ell_{j}^*[\xi']
\right\rangle_{\mathcal{F}}
\]
for any $\xi,\xi'\in \mathcal{F}$. Set $S_{\mathcal{F},j}=\ell_{j}+\ell_{j}^*$. Also, every element $b$ of $B$ naturally acts on $\mathcal{F}$ by left multiplication. 

Let $A_{\mathcal{F}}$ be a unital $*$-subalgebra of the unital algebra of all linear maps on $\mathcal{F}$ generated by $S_{\mathcal{F},j}$, $j\in[d]$, and $B$, and $E_{\mathcal{F}}$ a linear map defined by $E_{\mathcal{F}}[T]=\langle 1,T[1]\rangle_{\mathcal{F}}$ for every $T\in A_{\mathcal{F}}$. Then, we have the following:

\begin{Prop}\label{Prop_realization_Bsemicircular}
    The family $(S_{\mathcal{F},1},\dots,S_{\mathcal{F},d})$ is a $B$-free $B$-valued semi-circular system associated with $\eta_{1},\dots,\eta_{d}$ in the $B$-valued probability space $(A_{\mathcal{F}},B,E_{\mathcal{F}})$. 
\end{Prop}
See \cite[Chapter 4, section 4.6]{sp98} (and also \cite[section 2]{sh99}) for its details.

\section{$B$-valued Chebyshev family associated with $\eta$}\label{section_B_Chebyshev}
Let $\eta$ be a linear map from $B$ to $B$. The following object plays a central role in this paper:
\begin{Def}\label{Def_B_chebyshev}
    The \textit{$B$-valued Chebyshev family associated with $\eta$} is a family $\{U_{n}^{\eta}=U_{n}\}_{n\geq1}$ of $2n$-multilinear maps $U_{n}:(B\times B)^n\to B\langle X\rangle$ defined by 
    $U_{1}
    \left(

    \right)
    \,\middle|\,
    b_{1},\dots,b_{n},b_{1}',\dots,b_{n}'\in B
    \right\}.
    \]
\end{Def}

When $B=\mathbb{C}$ and $\eta=\mathrm{id}_{\mathbb{C}}$, the above definition recovers the usual Chebyshev polynomials of the second kind (see e.g., \cite{c78}). 

\begin{Rem}
    It is easy to see that $U_{n}$ is not only a $2n$-multilinear map, but also a partially $B$-balanced map in the sense that
    \begin{align*}
    \,&
    U_{n}
    \left(
    %
    \right)
    (X)c
    \end{align*}
    for any $a,b,b_{i},b_{i}',c\in B$ ($i\in[n]$) and any $j\in[n]$.
\end{Rem}

We have nice formulas of the free difference quotient with the $B$-valued Chebyshev family as follows.

\begin{Prop}\label{Prop_fromula_fdq_Chebyshev}
    We have 
    \begin{align*}
        \,&
        \partial_{X:B}
        \left[
        U_{n}
        \left(
        %
    \right)(X)
    \right]=b\otimes b'$. Next, for any $(b_{1},b_{1}'),(b_{2},b_{2}')\in B\times B$, we observe that
    \begin{align*}
        \partial_{X:B}
        \left[
        U_{2}
        \left(
        %
        \right)(X)
        b_{2}\otimes b_{2}'.
    \end{align*}
    
    Suppose that we have shown the desired formulas in the case of $1\leq k\leq n$ with an arbitrary $n\geq2$. For any $(b_{i},b_{i}')\in B\times B$ ($i\in[n+1]$), we observe, using the recursion formula in Definition \ref{Def_B_chebyshev}, that
    \begin{align*}
        \,&
        \partial_{X:B}
        \left[
        U_{n+1}
        \left(
        %
        \right)(X)
        \right].
    \end{align*}
    By the induction hypothesis, the right-hand side is equal to
    \begin{align*}
        \,&
        b_{1}\otimes b_{1}'
        U_{n}
        \left(
        %
        \right)(X)
        \otimes b_{n+1}',
    \end{align*}
    where the induction hypothesis was used in the second equality and the recursion formula in Definition \ref{Def_B_chebyshev} was used in the third equality. Thus, the proof has been completed by induction.
\end{proof}

\begin{Prop}\label{Prop_BChebyshev_tensor}
    Let $B$ be a unital $C^*$-algebra and $\eta_{1},\dots,\eta_{d}$ completely positive maps from $B$ to $B$, and $(A_{\mathcal{F}},B,E_{\mathcal{F}})$ and $(S_{\mathcal{F},i})_{i\in [d]}$ the realization of $B$-free $B$-valued semi-circular system associated with $\eta_{1},\dots,\eta_{d}$ (see Proposition \ref{Prop_realization_Bsemicircular}). Then, we have
    \begin{align*}
        \,&
        U_{n(1)}^{\eta_{i(1)}}
        \left(
        \begin{smallmatrix}
            b^{(1)}_{1}\\
            b^{(1)'}_{1}
        \end{smallmatrix}
        ;\cdots;
        \begin{smallmatrix}
            b^{(1)}_{n(1)}\\
            b^{(1)'}_{n(1)}
        \end{smallmatrix}
        \right)
        (S_{\mathcal{F},i(1)})
        \cdots
        U_{n(k)}^{\eta_{i(k)}}
        \left(
        \begin{smallmatrix}
            b^{(k)}_{1}\\
            b^{(k)'}_{1}
        \end{smallmatrix}
        ;\cdots;
        \begin{smallmatrix}
            b^{(k)}_{n(k)}\\
            b^{(k)'}_{n(k)}
         \end{smallmatrix}
        \right)
        (S_{\mathcal{F},i(k)})
        [1]\\
        &=
        \bigl(
        (b^{(1)}_{1}X_{i(1)}b^{(1)'}_{1})\cdots(b^{(1)}_{n(1)}X_{i(1)}b^{(1)'}_{n(1)})
        \bigr)\cdots
        \bigl(
        (b^{(k)}_{1}X_{i(k)}b^{(k)'}_{1})\cdots(b^{(k)}_{n(k)}X_{i(k)}b^{(k)'}_{n(k)})
        \bigr)
    \end{align*}
    for any $k\in\mathbb{N}$, any $n\in I(k,\mathbb{N})$, any $i\in\mathrm{Alt}(I(k,d))$ and $(b^{(j)}_{\ell},b^{(j)'}_{\ell})\in B\times B$ ($j\in[k]$ and $\ell\in[n(j)]$).
\end{Prop}
\begin{proof}
    We will give the proof only in the case of $k=1$ (and leave a comment for the case of $k=2$). The other general cases can also be treated similarly to the discussion below (\textit{n.b.}, the assumption of $i\in\mathrm{Alt}(I(k,d))$, that is, $i(j)\not=i(j+1)$, is essential in the general cases). By the definition of $S_{\mathcal{F}}$, it is easily seen that
    $U_{1}
    \left(
    \begin{smallmatrix}
        b\\
        b'
    \end{smallmatrix}
    \right)
    (S_{\mathcal{F}})[1]
    =bXb'$
    for all $b,b'\in B$. Moreover, we observe that
    \begin{align*}
        U_{2}
        \left(
        \begin{smallmatrix}
            b_{1}\\
            b_{1}'
        \end{smallmatrix}
        ;
        \begin{smallmatrix}
            b_{2}\\
            b_{2}'
        \end{smallmatrix}
        \right)
        (S_{\mathcal{F}})
        [1]
        &=
        (b_{1}S_{\mathcal{F}}b_{1}')(b_{2}S_{\mathcal{F}}b_{2}')[1]
        -b_{1}\eta(b_{1}'b_{2})b_{2}'[1]\\
        &=
        b_{1}S_{\mathcal{F}}b_{1}'\left[b_{2}Xb_{2}'\right]
        -b_{1}\eta(b_{1}'b_{2})b_{2}'\\
        &=b_{1}(\ell+\ell^*)[b_{1}'b_{2}Xb_{2}']
        -b_{1}\eta(b_{1}'b_{2})b_{2}'\\
        &=
        b_{1}Xb_{1}'b_{2}Xb_{2}'
        +b_{1}\eta(b_{1}'b_{2})b_{2}'
        -b_{1}\eta(b_{1}'b_{2})b_{2}'\\
        &=
        (b_{1}Xb_{1}')(b_{2}Xb_{2}')
    \end{align*}
    for all $b_{1},b_{1}',b_{2},b_{2}'\in B$. Assume that we have proved the desired formula in the case of $n\in[N]$ for some $N\geq2$. Then, we observe, using the recursion formula in Definition \ref{Def_B_chebyshev}, that
    \begin{align*}
        \,&
        U_{N+1}
        \left(
        \begin{smallmatrix}
            b_{1}\\
            b_{1}'
        \end{smallmatrix}
        ;\cdots;
        \begin{smallmatrix}
            b_{N+1}\\
            b_{N+1}'
        \end{smallmatrix}
        \right)
        (S_{\mathcal{F}})
        [1]\\
        &=
        U_{1}
        \left(
        \begin{smallmatrix}
            b_{1}\\
            b_{1}'
        \end{smallmatrix}
        \right)
        (S_{\mathcal{F}})
        U_{N}
        \left(
        \begin{smallmatrix}
            b_{2}\\
            b_{2}
        \end{smallmatrix}
        ;\cdots;
        \begin{smallmatrix}
            b_{N+1}\\
            b_{N+1}'
        \end{smallmatrix}
        \right)
        (S_{\mathcal{F}})[1]
        -b_{1}\eta(b_{1}'b_{2})b_{2}'
        U_{N-1}
        \left(
        \begin{smallmatrix}
            b_{3}\\
            b_{3}'
        \end{smallmatrix}
        ;\cdots;
        \begin{smallmatrix}
            b_{N+1}\\
            b_{N+1}'
        \end{smallmatrix}
        \right)
        (S_{\mathcal{F}})[1]\\
        &=
        b_{1}S_{\mathcal{F}}b_{1}'
        \left[
        (b_{2}Xb_{2}')
        \cdots(b_{N+1}Xb_{N+1}')
        \right]
        -
        b_{1}\eta(b_{1}'b_{2})b_{2}'
        (b_{3}Xb_{3}')\cdots (b_{N+1}Xb_{N+1}')\\
        &=
        (b_{1}Xb_{1}')(b_{2}Xb_{2}')
        \cdots(b_{N+1}Xb_{N+1}')
        +b_{1}\eta(b_{1}'b_{2})b_{2}'(b_{3}Xb_{3}')\cdots (b_{N+1}Xb_{N+1}')\\
        &\quad
        -
        b_{1}\eta(b_{1}'b_{2})b_{2}'
        (b_{3}Xb_{3}')\cdots (b_{N+1}Xb_{N+1}')\\
        &=
        (b_{1}Xb_{1}')(b_{2}Xb_{2}')
        \cdots(b_{N+1}Xb_{N+1}').
    \end{align*}
    Hence, the proof has been completed by induction.

    We can also see the case of $k=2$ by the same computation as the case of $k=1$, since 
    \[\ell_{i(1)}(b_{0}X_{i(2)}b_{1}X_{i(2)}b_{3}\cdots X_{i(2)}b_{n})=X_{i(1)}b_{0}X_{i(2)}b_{1}X_{i(2)}b_{3}\cdots X_{i(2)}b_{n},
    \]
    \[
    \ell_{i(1)}^*(b_{0}X_{i(2)}b_{1}X_{i(2)}b_{3}\cdots X_{i(2)}b_{n})
    =\delta_{i(1),i(2)}\eta_{i(1)}(b_{0})b_{1}X_{i(2)}b_{3}\cdots X_{i(2)}b_{n}=0
    \] for any $i\in\mathrm{Alt}(I(2,d))$ (that is, $i(1)\not=i(2)$).
\end{proof}

Using Proposition \ref{Prop_BChebyshev_tensor}, we have the following corollary:
\begin{Cor}\label{cor_orthogonality_chebyshev}
   Under the same notation as Proposition \ref{Prop_Bsemicircular_equivalent}, the $B$-valued Chebyshev family associated with $\eta=(\eta_{1},\dots,\eta_{d})$ are orthogonal with respect to $B$-free $B$-valued semi-circular system $S=(S_{1},\dots,S_{d})$. Namely, for any $M,N\in\mathbb{N}$, $m\in I(M,\mathbb{N})$, $n\in I(N,\mathbb{N})$ and $i\in\mathrm{Alt}(I(M,d))$, $j\in\mathrm{Alt}(I(N,d))$, if $m\not=n$ or $i\not=j$, then 
   \begin{align*}
   \,&E\Bigl[
   \left(U^{\eta_{i(M)}}_{m(M)}
   \left(
   \begin{smallmatrix}
       b_{1,M}\\
       b'_{1,M}
   \end{smallmatrix}
   ;\dots;
   \begin{smallmatrix}
       b_{m(M),M}\\
       b'_{m(M),M}
   \end{smallmatrix}
   \right)(S_{i(M)})\cdots
   U^{\eta_{i(1)}}_{m(1)}
   \left(
   \begin{smallmatrix}
       b_{1,1}\\
       b'_{1,1}
   \end{smallmatrix}
   ;\dots;
   \begin{smallmatrix}
       b_{m(1),1}\\
       b'_{m(1),1}
   \end{smallmatrix}
   \right)(S_{i(1)})\right)^*\\
   &\quad\times
   U^{\eta_{j(1)}}_{n(1)}
   \left(
   \begin{smallmatrix}
       c_{1,1}\\
       c'_{1,1}
   \end{smallmatrix}
   ;\dots;
   \begin{smallmatrix}
       c_{n(1),1}\\
       c'_{n(1),1}
   \end{smallmatrix}
   \right)(S_{j(1)})
   \cdots
   U^{\eta_{j(N)}}_{n(N)}
   \left(
   \begin{smallmatrix}
       c_{1,N}\\
       c_{1,N}'
   \end{smallmatrix}
   ;\dots;
   \begin{smallmatrix}
       c_{n(N),N}\\
       c_{n(N),N}'
   \end{smallmatrix}
   \right)(S_{j(N)})
   \Bigr]=0.
   \end{align*}
\end{Cor}

\begin{Lem}\label{Lem_spann_BChebyshev}
    The space $B\langle X\rangle$ is spanned by $B$ and the $B$-valued Chebyshev family $\{U^{\eta}_{n}\}_{n\geq1}$ associated with $\eta$ as a $B$-bimodule.
\end{Lem}
\begin{proof}
    It is clear that $B\oplus BXB=\mathrm{span}\{B,\mathrm{ran}(U_{1})\}$, where $BXB=\mathrm{span}\{bXb'\,|\,b,b'\in B\}$. Assume that we have shown that $B\oplus\bigoplus_{1\leq k\leq m}B(XB)^{k}=\mathrm{span}\{B,\mathrm{ran}(U_{1}),\dots,\mathrm{ran}(U_{m})\}$ for any $1\leq m\leq N$. Then, we have to see that
    \[
    (b_{1}Xb_{1}')\cdots (b_{N+1}Xb_{N+1}')\in\mathrm{span}\{B,\mathrm{ran}(U_{1}),\dots,\mathrm{ran}(U_{N+1})\}
    \]
    for any $b_{1},b_{1}',\dots,b_{N+1},b_{N+1}'\in B$. By the induction hypothesis, there exist $j_{1},\dots,j_{m}\in[N]$ (with $j_{1}\leq\cdots\leq j_{m}$) and families $(b^{(\ell)}_{s};1\leq s\leq j_{\ell})$, $\ell\in[m]$, of elements of $B$ such that
    \[
    (b_{2}Xb_{2}')\cdots(b_{N+1}Xb_{N+1}')
    =
    \sum_{1\leq \ell\leq m}
    U_{j_{\ell}}
    \left(
    \begin{smallmatrix}
        b^{(\ell)}_{1}\\
        b^{(\ell)'}_{1}
    \end{smallmatrix}
    ;\cdots;
    \begin{smallmatrix}
        b^{(\ell)}_{j_{\ell}}\\
        b^{(\ell)'}_{j_{\ell}}
    \end{smallmatrix}
    \right)(X).
    \]
    By $b_{1}Xb_{1}'=U_{1}\left(\begin{smallmatrix}b_{1}\\ b_{1}'\end{smallmatrix}\right)(X)$ and the recursion formula in Definition \ref{Def_B_chebyshev}, we observe that
    \begin{align*}
        \,&(b_{1}Xb_{1}')(b_{2}Xb_{2}')\cdots (b_{N+1}Xb_{N+1}')\\
        &=
        \sum_{1\leq \ell\leq m}
        U_{1}
        \left(
        \begin{smallmatrix}
            b_{1}\\ 
            b_{1}'
        \end{smallmatrix}\right)(X)
        U_{j_{\ell}}
        \left(
        \begin{smallmatrix}
            b^{(\ell)}_{1}\\
            b^{(\ell)'}_{1}
        \end{smallmatrix}
        ;\cdots;
        \begin{smallmatrix}
            b^{(\ell)}_{j_{\ell}}\\
            b^{(\ell)'}_{j_{\ell}}
        \end{smallmatrix}
        \right)(X)\\
        &=
        \sum_{1\leq \ell\leq m}
        U_{j_{\ell}+1}
        \left(
        \begin{smallmatrix}
            b_{1}\\
            b_{1}'
        \end{smallmatrix};
        \begin{smallmatrix}
            b^{(\ell)}_{1}\\
            b^{(\ell)'}_{1}
        \end{smallmatrix}
        ;\cdots;
        \begin{smallmatrix}
            b^{(\ell)}_{j_{\ell}}\\
            b^{(\ell)'}_{j_{\ell}}
        \end{smallmatrix}
        \right)(X)
        +b_{1}\eta(b_{1}'b^{(\ell)}_{1})b^{(\ell)'}_{1}
        U_{j_{\ell}-1}
        \left(
        \begin{smallmatrix}
            b^{(\ell)}_{2}\\
            b^{(\ell)'}_{2}
        \end{smallmatrix}
        ;\cdots;
        \begin{smallmatrix}
            b^{(\ell)}_{j_{\ell}}\\
            b^{(\ell)'}_{j_{\ell}}
        \end{smallmatrix}
        \right)(X).
    \end{align*}
    The right-hand side is in $\mathrm{span}\{B,\mathrm{ran}(U_{1}),\dots,\mathrm{ran}(U_{N+1})\}$, since $j_{1},\dots,j_{m}\in[N]$, and hence the proof has been completed by induction.
\end{proof}

By Propositions \ref{Prop_realization_Bsemicircular} and \ref{Prop_BChebyshev_tensor} and Lemma \ref{Lem_spann_BChebyshev}, we have the following proposition:

\begin{Prop}\label{Prop_Bsemicircular_equivalent}
    Let $(A,B,E)$ be a $B$-valued probability space with completely positive maps $\eta_{1},\dots,\eta_{d}$ from $B$ to $B$. Let $S_{1},\dots,S_{d}$ be self-adjoint $B$-valued random variables with mean $0$ and variance $\eta_{1},\dots,\eta_{d}$, respectively, in $(A,B,E)$. Then, the following are equivalent:
    \begin{enumerate}
        \item $S=(S_{1},\dots,S_{d})$ is a $B$-free $B$-valued semi-circular system associated with $\eta_{1},\dots,\eta_{d}$.
        \item We have
        \[
        E
        \left[
        U_{n(1)}^{\eta_{i(1)}}
        \left(
        \begin{smallmatrix}
            b^{(1)}_{1}\\
            b^{(1)'}_{1}
        \end{smallmatrix}
        ;\cdots;
        \begin{smallmatrix}
            b^{(1)}_{n(1)}\\
            b^{(1)'}_{n(1)}
        \end{smallmatrix}
        \right)
        (S_{i(1)})
        \cdots
        U_{n(k)}^{\eta_{i(k)}}
        \left(
        \begin{smallmatrix}
            b^{(k)}_{1}\\
            b^{(k)'}_{1}
        \end{smallmatrix}
        ;\cdots;
        \begin{smallmatrix}
            b^{(k)}_{n(k)}\\
            b^{(k)'}_{n(k)}
         \end{smallmatrix}
        \right)
        (S_{i(k)})
        \right]
        =0
        \]
        for any $k\in\mathbb{N}$, $n\in I(k,\mathbb{N})$, any $i\in \mathrm{Alt}(I(k,d))$ and any $(b^{(j)}_{\ell},b^{(j)'}_{\ell})\in B\times B$ ($j\in[k]$ and $\ell\in[n(j)]$).
    \end{enumerate}
\end{Prop}
\begin{proof}
    (1)$\Rightarrow$(2): Suppose (1), that is, let $S=(S_{1},\dots,S_{d})$ be a $B$-free $B$-valued semi-circular system associated to $\eta_{1},\dots,\eta_{d}$ in $A$. Then, we have $E[P(S)]=E_{\mathcal{F}}[P(S_{\mathcal{F}})]=\left\langle1, P(S_{\mathcal{F}})[1]\right\rangle_{\mathcal{F}}$ for any $P(X)\in B\langle X_{1},\dots,X_{d}\rangle$.
    In particular, we have
    \begin{align*}
        \,&
        E\left[
        U_{n(1)}^{\eta_{i(1)}}
        \left(
        \begin{smallmatrix}
            b^{(1)}_{1}\\
            b^{(1)'}_{1}
        \end{smallmatrix}
        ;\cdots;
        \begin{smallmatrix}
            b^{(1)}_{n(1)}\\
            b^{(1)'}_{n(1)}
        \end{smallmatrix}
        \right)
        (S_{i(1)})
        \cdots
        U_{n(k)}^{\eta_{i(k)}}
        \left(
        \begin{smallmatrix}
            b^{(k)}_{1}\\
            b^{(k)'}_{1}
        \end{smallmatrix}
        ;\cdots;
        \begin{smallmatrix}
            b^{(k)}_{n(k)}\\
            b^{(k)'}_{n(k)}
         \end{smallmatrix}
        \right)
        (S_{i(k)})
        \right]\\
        &=
        \left\langle
        1, 
        U_{n(1)}^{\eta_{i(1)}}
        \left(
        \begin{smallmatrix}
            b^{(1)}_{1}\\
            b^{(1)'}_{1}
        \end{smallmatrix}
        ;\cdots;
        \begin{smallmatrix}
            b^{(1)}_{n(1)}\\
            b^{(1)'}_{n(1)}
        \end{smallmatrix}
        \right)
        (S_{i(1)})
        \cdots
        U_{n(k)}^{\eta_{i(k)}}
        \left(
        \begin{smallmatrix}
            b^{(k)}_{1}\\
            b^{(k)'}_{1}
        \end{smallmatrix}
        ;\cdots;
        \begin{smallmatrix}
            b^{(k)}_{n(k)}\\
            b^{(k)'}_{n(k)}
         \end{smallmatrix}
        \right)
        (S_{i(k)})
        [1]\right\rangle_{\mathcal{F}}
    \end{align*}
    for any $k\in\mathbb{N}$, $n\in I(k,\mathbb{N})$, any $i\in \mathrm{Alt}(I(k,d))$ and any $(b^{(j)}_{\ell},b^{(j)'}_{\ell})\in B\times B$ ($j\in[k]$ and $\ell\in[n(j)]$).
    By Propositions \ref{Prop_BChebyshev_tensor} and the definition of the $B$-valued pre-inner product $\langle\cdot,\cdot\rangle_{\mathcal{F}}$, the right-hand side is equal to
    \[
    \left\langle
    1,
    \bigl((b^{(1)}_{1}X_{i(1)}b^{(1)'}_{1})\cdots(b^{(1)}_{n(1)}X_{i(1)}b^{(1)'}_{n(1)})\bigr)
    \cdots
    \bigl((b^{(k)}_{1}X_{i(k)}b^{(k)'}_{1})\cdots(b^{(k)}_{n(k)}X_{i(k)}b^{(k)'}_{n(k)})\bigr)
    \right\rangle_{\mathcal{F}}
    =0,
    \]
    and hence we have obtained that
    \[
    E\left[
        U_{n(1)}^{\eta_{i(1)}}
        \left(
        \begin{smallmatrix}
            b^{(1)}_{1}\\
            b^{(1)'}_{1}
        \end{smallmatrix}
        ;\cdots;
        \begin{smallmatrix}
            b^{(1)}_{n(1)}\\
            b^{(1)'}_{n(1)}
        \end{smallmatrix}
        \right)
        (S_{i(1)})
        \cdots
        U_{n(k)}^{\eta_{i(k)}}
        \left(
        \begin{smallmatrix}
            b^{(k)}_{1}\\
            b^{(k)'}_{1}
        \end{smallmatrix}
        ;\cdots;
        \begin{smallmatrix}
            b^{(k)}_{n(k)}\\
            b^{(k)'}_{n(k)}
         \end{smallmatrix}
        \right)
        (S_{i(k)})
    \right]=0.
    \]

    (2)$\Rightarrow$(1):
    Suppose (2). By Lemma \ref{Lem_spann_BChebyshev}, for any $P(X_{j})\in B\langle X_{j}\rangle$, $j\in[d]$, there exists $n\in I(k,\mathbb{N})$ and elements $b,(b^{(j)}_{i_{j}},b^{(j)'}_{i_{j}}\,|\,j\in[k],i_{j}\in[n(j)])$ of $B$ such that
    \[
    P(X_{j})
    =
    b+
    \sum_{j=1}^{k}
    U_{n(j)}^{\eta_{j}}
    \left(
    \begin{smallmatrix}
        b^{(j)}_{1}\\
        b^{(j)'}_{1}
    \end{smallmatrix}
    ;\cdots;
    \begin{smallmatrix}
        b^{(j)}_{n(j)}\\
        b^{(j)'}_{n(j)}
    \end{smallmatrix}
    \right)
    (X_{j}).
    \]
    By condition (2), we have
    \begin{align*}
        E\left[
        P(S_{j})
        \right]
        &=
        b+
        \sum_{j=1}^{k}
        E
        \left[
        U_{n(j)}
        \left(
        \begin{smallmatrix}
            b^{(j)}_{1}\\
            b^{(j)'}_{1}
        \end{smallmatrix}
        ;\cdots;
        \begin{smallmatrix}
            b^{(j)}_{n(j)}\\
            b^{(j)'}_{n(j)}
        \end{smallmatrix}
        \right)
        (S_{j})
        \right]
        =b.
    \end{align*}
    On the other hand, by Proposition \ref{Prop_BChebyshev_tensor}, we also have
    \[
    E_{\mathcal{F}}
    \left[
    P(S_{\mathcal{F},j})
    \right]
    =b+
    \sum_{j=1}^{k}
    \left\langle
    1,
    U_{n(j)}
    \left(
    \begin{smallmatrix}
        b^{(j)}_{1}\\
        b^{(j)'}_{1}
    \end{smallmatrix}
    ;\cdots;
    \begin{smallmatrix}
        b^{(j)}_{n(j)}\\
        b^{(j)'}_{n(j)}
    \end{smallmatrix}
    \right)
    (S_{\mathcal{F},j})[1]
    \right\rangle_{\mathcal{F}}
    =b,
    \]
    that is, $E\left[P(S_{j})\right]=E_{\mathcal{F}}\left[P(S_{\mathcal{F},j})\right]$ for all $P(X_{j})\in B\langle X_{j}\rangle$. Since $S_{\mathcal{F},j}$ is a $B$-valued semi-circular element, so is $S_{j}$, $j\in[d]$. The rest is to show the $B$-freeness of $(S_{j})_{j\in[d]}$.
    
    Let $P_{1}(X_{i(1)}),\dots,P_{k}(X_{i(k)})$ be elements in $B_{\langle d\rangle}$ such that $E[P_{j}(S_{i(j)})]=0$ with $i(j)\not=i(j+1)$ ($j\in[k-1]$), that is, $i\in\mathrm{Alt}(I(k,d))$. What we have shown above says that
    \begin{align*}
        P_{j}(X_{i(j)})
        =
        \sum_{1\leq \ell\leq N(j)}
        U^{\eta_{i(j)}}_{n_{j}(\ell)}
        \left(
        \begin{smallmatrix}
            b^{(\ell,j)}_{1}\\
            b^{(\ell,j)'}_{1}
        \end{smallmatrix}
        ;\cdots;
        \begin{smallmatrix}
            b^{(\ell,j)}_{n_{j}(\ell)}\\
            b^{(\ell,j)'}_{n_{j}(\ell)}
        \end{smallmatrix}
        \right)
        (X_{i(j)})
    \end{align*}
    for some $N\in I(k,\mathbb{N})$, some $n_{j}(\cdot)\in I(N(j),\mathbb{N})$ with $n_{j}(\ell)\leq\mathrm{deg}(P_{j})$ ($\ell\in[N(j)]$) and elements $\{b^{(\ell,j)}_{t}\,|\,j\in[k],\ell\in[N(j)],t\in[n_{j}(\ell)]\}$ of $B$. 
    Then, we have 
    \begin{align*}
        \,&
        P_{1}(X_{i(1)})\cdots P_{k}(X_{i(k)})\\
        &=
        \sum_{\substack{1\leq \ell_{1}\leq N(1)\\\vdots\\1\leq \ell_{k}\leq N(k)}}
        U^{\eta_{i(1)}}_{n_{1}(\ell_{1})}
        \left(
        \begin{smallmatrix}
            b^{(\ell,1)}_{1}\\
            b^{(\ell,1)'}_{1}
        \end{smallmatrix}
        ;\cdots;
        \begin{smallmatrix}
            b^{(\ell,1)}_{n_{1}(\ell_{1})}\\
            b^{(\ell,1)'}_{n_{1}(\ell_{1})}
        \end{smallmatrix}
        \right)
        (X_{i(1)})
        \cdots
        U^{\eta_{i(k)}}_{n_{k}(\ell_{k})}
        \left(
        \begin{smallmatrix}
            b^{(\ell,k)}_{1}\\
            b^{(\ell,k)'}_{1}
        \end{smallmatrix}
        ;\cdots;
        \begin{smallmatrix}
            b^{(\ell,k)}_{n_{k}(\ell_{k})}\\
            b^{(\ell,k)'}_{n_{k}(\ell_{k})}
        \end{smallmatrix}
        \right)
        (X_{i(k)}).
    \end{align*}
    Since $n_{1}(\ell_{1}),\dots,n_{k}(\ell_{k})\geq1$ and $i\in \mathrm{Alt}(I(k,d))$, we have
    \begin{align*}
        \,&
        E\left[P_{1}(S_{i(1)})\cdots P_{k}(S_{i(k)})\right]\\
        &=
        \sum_{\substack{1\leq \ell_{1}\leq N(1)\\\vdots\\1\leq \ell_{k}\leq N(k)}}
        E\left[
        U^{\eta_{i(1)}}_{n_{1}(\ell_{1})}
        \left(
        \begin{smallmatrix}
            b^{(\ell,1)}_{1}\\
            b^{(\ell,1)'}_{1}
        \end{smallmatrix}
        ;\cdots;
        \begin{smallmatrix}
            b^{(\ell,1)}_{n_{1}(\ell_{1})}\\
            b^{(\ell,1)'}_{n_{1}(\ell_{1})}
        \end{smallmatrix}
        \right)
        (S_{i(1)})
        \cdots
        U^{\eta_{i(k)}}_{n_{k}(\ell_{k})}
        \left(
        \begin{smallmatrix}
            b^{(\ell,k)}_{1}\\
            b^{(\ell,k)'}_{1}
        \end{smallmatrix}
        ;\cdots;
        \begin{smallmatrix}
            b^{(\ell,k)}_{n_{k}(\ell_{k})}\\
            b^{(\ell,k)'}_{n_{k}(\ell_{k})}
        \end{smallmatrix}
        \right)
        (S_{i(k)})
        \right]\\
        &=0
    \end{align*}
    by assumption (2). Therefore, $(S_{j})_{j\in[d]}$ are $B$-freely independent.
\end{proof}

\section{The conjugate variable associated with $\eta$ of $B$-valued semi-circular element}\label{section_conjugate}

In the rest of this paper, we will work in the following setting: Let $B\subset A$ be a unital inclusion of unital $C^*$-algebras with a faithful tracial state $\tau$ on $A$ and a $\tau$-preserving conditional expectation $E$ from $A$ onto $B$, which are denoted by $(A,B,\tau,E)$. Let $\eta$ be a completely positive map from $B$ to $B$.

Here, we will give some new notations. We denote by $\langle\cdot,\cdot\rangle_{\tau}$ the inner product on $A$ defined by $\left\langle a_{1},a_{2}\right\rangle_{\tau}=\tau\left(a_{1}^*a_{2}\right)$ for all $a_{1},a_{2}\in A$ and $\|a\|_{\tau}=\langle a,a\rangle_{\tau}^{\frac{1}{2}}$ for all $a\in A$. We denote by $\langle\cdot,\cdot\rangle_{\eta}$ the pre-inner product on $A\otimes_{\mathrm{alg}} A$ given by
\[
\left\langle
a_{1}\otimes a_{2},
a_{3}\otimes a_{4}
\right\rangle_{\eta}
=\tau\left(
a_{2}^*
\eta\left(E\left[a_{1}^*a_{3}\right]\right)a_{4}
\right)
\]
for any $a_{1},a_{2},a_{3},a_{4}\in A$ and $\|\xi\|_{\eta}=\langle\xi,\xi\rangle_{\eta}^{\frac{1}{2}}$ for all $\xi\in A\otimes_{\mathrm{alg}}A$. (The positivity of $\langle\cdot,\cdot\rangle_{\eta}$ follows from the complete positivity of $\eta$.) Its Hilbert space separation is also denoted by $\langle\cdot,\cdot\rangle_{\eta}$.

The following proposition is a $B$-valued semi-circular analogue of Stein's equation. This is essentially same as \cite[Proposition 3.10]{sh00} and is characterized in terms of Speicher's $B$-valued free cumulants \cite[Example 2.5]{l24}, but we will give a different proof as a warm-up exercises of the use of the $B$-valued Chebyshev family. 

\begin{Prop}\label{Prop_Bsemicircular_Stein}
    Let $\eta_{1},\dots,\eta_{d}$ be completely positive maps from $B$ to $B$. Let $(S_{1},\dots,S_{d})$ be a $B$-free $B$-valued semi-circular system associated with $\eta_{1},\dots,\eta_{d}$ in $(A,B,E,\tau)$. Then, we have
    \[
    \left\langle
    S_{j},P(S_{1},\dots,S_{d})
    \right\rangle_{\tau}
    =
    \left\langle
    1\otimes1,
    \mathrm{ev}_{S}^{\otimes2}
    \left(
    \partial_{j}\left[P(X_{1},\dots,X_{d})\right]
    \right)
    \right\rangle_{\eta_{j}}
    \]
    for any $j\in[d]$ and $P(X_{1},\dots,X_{d})\in B\langle X_{1},\dots,X_{d}\rangle$.
\end{Prop}
\begin{proof}
    By Lemma \ref{Lem_spann_BChebyshev}, it suffices to see the case when 
    \begin{align*}
        \,&P(X_{1},\dots,X_{d})
        =
        U^{\eta_{i(1)}}_{n(1)}
        \left(
        \begin{smallmatrix}
            b^{(1)}_{1}\\
            b^{(1)'}_{1}
        \end{smallmatrix}
        ;\cdots;
        \begin{smallmatrix}
            b^{(1)}_{n(1)}\\
            b^{(1)'}_{n(1)}
        \end{smallmatrix}
        \right)
        (X_{i(1)})
        \cdots
        U_{n(k)}^{\eta_{i(k)}}
        \left(
        \begin{smallmatrix}
            b^{(k)}_{1}\\
            b^{(k)'}_{1}
        \end{smallmatrix}
        ;\cdots;
        \begin{smallmatrix}
            b^{(k)}_{n(k)}\\
            b^{(k)'}_{n(k)}
         \end{smallmatrix}
        \right)
        (X_{i(k)})
    \end{align*}
    for any $k\in\mathbb{N}$, any $b_{1},\dots,b_{k}\in B$, $n\in I(k,\mathbb{N})$ and $i\in \mathrm{Alt}(I(k,d))$.
    We observe that
    \begin{align*}
        \,&
        \left\langle
        S_{j},P(S_{1},\dots,S_{d})
        \right\rangle_{\tau}\\
        &=
        \tau\left(
        U_{1}^{\eta_{j}}
        \left(
        \begin{smallmatrix}
            1\\
            1
        \end{smallmatrix}
        \right)
        (S_{j})
        U_{n(1)}^{\eta_{i(1)}}
        \left(
        \begin{smallmatrix}
            b^{(1)}_{1}\\
            b^{(1)'}_{1}
        \end{smallmatrix}
        ;\cdots;
        \begin{smallmatrix}
            b^{(1)}_{n(1)}\\
            b^{(1)'}_{n(1)}
        \end{smallmatrix}
        \right)
        (S_{i(1)})
        \cdots
        U_{n(k)}^{\eta_{i(k)}}
        \left(
        \begin{smallmatrix}
            b^{(k)}_{1}\\
            b^{(k)'}_{1}
        \end{smallmatrix}
        ;\cdots;
        \begin{smallmatrix}
            b^{(k)}_{n(k)}\\
            b^{(k)'}_{n(k)}
         \end{smallmatrix}
        \right)
        (S_{i(k)})
        \right).
    \end{align*}
    Assume that $k\geq2$ or $n(1)\geq2$. If $j\not=i(1)$, then the right-hand side is equal to $0$ by Proposition \ref{Prop_Bsemicircular_equivalent}. If $j=i(1)$, then the recursion formula in Definition \ref{Def_B_chebyshev} and Proposition \ref{Prop_Bsemicircular_equivalent} enable us to confirm that the right-hand side is equal to
    \begin{align*}
        \,&
        \tau\left(
        U_{n(1)+1}^{\eta_{i(1)}}
        \left(
        \begin{smallmatrix}
            1\\
            1
        \end{smallmatrix}
        ;
        \begin{smallmatrix}
            b^{(1)}_{1}\\
            b^{(1)'}_{1}
        \end{smallmatrix}
        ;\cdots;
        \begin{smallmatrix}
            b^{(1)}_{n(1)}\\
            b^{(1)'}_{n(1)}
        \end{smallmatrix}
        \right)
        (S_{i(1)})
        \cdots
        U_{n(k)}^{\eta_{i(k)}}
        \left(
        \begin{smallmatrix}
            b^{(k)}_{1}\\
            b^{(k)'}_{1}
        \end{smallmatrix}
        ;\cdots;
        \begin{smallmatrix}
            b^{(k)}_{n(k)}\\
            b^{(k)'}_{n(k)}
         \end{smallmatrix}
        \right)
        (S_{i(k)})
        \right)\\
        &\quad
        +
        \tau\left(
        \eta_{i(1)}
        \left(b^{(1)}_{1}\right)
        b^{(1)'}_{1}
        U_{n(1)-1}^{\eta_{i(1)}}
        \left(
        \begin{smallmatrix}
            b^{(1)}_{2}\\
            b^{(1)'}_{2}
        \end{smallmatrix}
        ;\cdots;
        \begin{smallmatrix}
            b^{(1)}_{n(1)}\\
            b^{(1)'}_{n(1)}
        \end{smallmatrix}
        \right)
        (S_{i(1)})
        \cdots
        U_{n(k)}^{\eta_{i(k)}}
        \left(
        \begin{smallmatrix}
            b^{(k)}_{1}\\
            b^{(k)'}_{1}
        \end{smallmatrix}
        ;\cdots;
        \begin{smallmatrix}
            b^{(k)}_{n(k)}\\
            b^{(k)'}_{n(k)}
         \end{smallmatrix}
        \right)
        (S_{i(k)})
        \right)\\
        &=0.
    \end{align*}
    Thus, $\left\langle S_{j},P(S_{1},\dots,S_{d})\right\rangle_{\tau}=0$ when $k\geq2$ or $n(1)\geq2$. If $k=1$ and $n(1)=1$, then we observe that
    \begin{align*}
        \,&
        \left\langle
        S_{j},
        U_{1}
        \left(
        \begin{smallmatrix}
            b^{(1)}_{1}\\
            b^{(1)'}_{1}
        \end{smallmatrix}
        \right)
        (S_{i(1)})
        \right\rangle_{\tau}=
        \tau\left(
        S_{j}
        U_{1}
        \left(
        \begin{smallmatrix}
            b^{(1)}_{1}\\
            b^{(1)'}_{1}
        \end{smallmatrix}
        \right)
        (S_{i(1)})
        \right)
        =\delta_{j,i(1)}\tau\left(\eta_{i(1)}\left(b^{(1)}_{1}\right)b^{(1)'}_{1}\right)
    \end{align*}
    by $B$-free independence. (Here, note that the variance of $S_{i(1)}$ is given by $\eta_{i(1)}$.)
    Consequently, we have obtained that
    \[
    \left\langle
    S_{j},P(S_{1},\dots,S_{d})
    \right\rangle_{\tau}
    =
    \begin{cases}
        0;\quad \mbox{if }k\geq2\mbox{ or }n(1)\geq2,\\
        \delta_{j,i(1)}\tau\left(\eta_{i(1)}\left(b^{(1)}_{1}\right)b^{(1)'}_{1}\right);\quad \mbox{if }k=1\mbox{ and }n(1)=1.
    \end{cases}
    \]

    On the other hand, by Proposition \ref{Prop_fromula_fdq_Chebyshev}, we observe that
    \begin{align*}
        \,&
        \left\langle
        1\otimes1,
        \mathrm{ev}_{S}^{\otimes2}
        \left(
        \partial_{j}\left[P(X_{1},\dots,X_{d})\right]
        \right)
        \right\rangle_{\eta_{j}}\\
        &=
        \sum_{\ell=1}^{k}
        \delta_{j,i(\ell)}
        \Biggl\langle
        1\otimes1,
        \left(
        U_{n(1)}^{\eta_{i(1)}}
        \left(
        \begin{smallmatrix}
            b^{(1)}_{1}\\
            b^{(1)'}_{1}
        \end{smallmatrix}
        ;\cdots;
        \begin{smallmatrix}
            b^{(1)}_{n(1)}\\
            b^{(1)'}_{n(1)}
        \end{smallmatrix}
        \right)
        (S_{i(1)})
        \cdots
        U_{n(\ell-1)}^{\eta_{i(\ell-1)}}
        \left(
        \begin{smallmatrix}
            b^{(\ell-1)}_{1}\\
            b^{(\ell-1)'}_{1}
        \end{smallmatrix}
        ;\cdots;
        \begin{smallmatrix}
            b^{(\ell-1)}_{n(\ell-1)}\\
            b^{(\ell-1)}_{n(\ell-1)}
        \end{smallmatrix}
        \right)
        (S_{i(\ell-1)})
        \otimes1
        \right)\\
        &\quad
        \times\mathrm{ev}_{S}^{\otimes2}
        \left(
        \partial_{j}
        \left[
        U_{n(\ell)}^{\eta_{i(\ell)}}
        \left(
        \begin{smallmatrix}
            b^{(\ell)}_{1}\\
            b^{(\ell)'}_{1}
        \end{smallmatrix}
        ;\cdots;
        \begin{smallmatrix}
            b^{(\ell)}_{n(\ell)}\\
            b^{(\ell)'}_{n(\ell)}
        \end{smallmatrix}
        \right)
        (S_{i(\ell)})
        \right]
        \right)\\
        &\quad
        \times\left(
        1\otimes
        U_{n(\ell+1)}^{\eta_{i(\ell+1)}}
        \left(
        \begin{smallmatrix}
            b^{(\ell+1)}_{1}\\
            b^{(\ell+1)'}_{1}
        \end{smallmatrix}
        ;\cdots;
        \begin{smallmatrix}
            b^{(\ell+1)}_{n(\ell+1)}\\
            b^{(\ell+1)}_{n(\ell+1)}
        \end{smallmatrix}
        \right)
        (S_{i(\ell+1)})
        \cdots
        U_{n(k)}^{\eta_{i(k)}}
        \left(
        \begin{smallmatrix}
            b^{(k)}_{1}\\
            b^{(k)'}_{1}
        \end{smallmatrix}
        ;\cdots;
        \begin{smallmatrix}
            b^{(k)}_{n(k)}\\
            b^{(k)'}_{n(k)}
        \end{smallmatrix}
        \right)
        (S_{i(k)})
        \right)
        \Biggr\rangle_{\eta_{j}}\\
        &
        =\sum_{\ell=1}^{k}
        \delta_{j,i(\ell)}
        \Biggl\langle
        1\otimes1,
        U_{n(1)}^{\eta_{i(1)}}
        \left(
        \begin{smallmatrix}
            b^{(1)}_{1}\\
            b^{(1)'}_{1}
        \end{smallmatrix}
        ;\cdots;
        \begin{smallmatrix}
            b^{(1)}_{n(1)}\\
            b^{(1)'}_{n(1)}
        \end{smallmatrix}
        \right)
        (S_{i(1)})
        \cdots
        U_{n(\ell-1)}^{\eta_{i(\ell-1)}}
        \left(
        \begin{smallmatrix}
            b^{(\ell-1)}_{1}\\
            b^{(\ell-1)'}_{1}
        \end{smallmatrix}
        ;\cdots;
        \begin{smallmatrix}
            b^{(\ell-1)}_{n(\ell-1)}\\
            b^{(\ell-1)}_{n(\ell-1)}
        \end{smallmatrix}
        \right)
        (S_{i(\ell-1)})b^{(\ell)}_{1}\\
        &\quad\quad\quad\otimes
        U_{n(\ell)-1}^{\eta_{i(\ell)}}
        \left(
        \begin{smallmatrix}
            b^{(\ell)'}_{1}b^{(\ell)}_{2}\\
            b^{(\ell)'}_{2}
        \end{smallmatrix}
        ;\cdots;
        \begin{smallmatrix}
            b^{(\ell)}_{n(\ell)}\\
            b^{(\ell)'}_{n(\ell)}
        \end{smallmatrix}
        \right)
        (S_{i(\ell)})
        U_{n(\ell+1)}^{\eta_{i(\ell+1)}}
        \left(
        \begin{smallmatrix}
            b^{(\ell+1)}_{1}\\
            b^{(\ell+1)'}_{1}
        \end{smallmatrix}
        ;\cdots;
        \begin{smallmatrix}
            b^{(\ell+1)}_{n(\ell+1)}\\
            b^{(\ell+1)}_{n(\ell+1)}
        \end{smallmatrix}
        \right)
        (S_{i(\ell+1)})\\
        &\hspace{8cm}
        \cdots
        U_{n(k)}^{\eta_{i(k)}}
        \left(
        \begin{smallmatrix}
            b^{(k)}_{1}\\
            b^{(k)'}_{1}
        \end{smallmatrix}
        ;\cdots;
        \begin{smallmatrix}
            b^{(k)}_{n(k)}\\
            b^{(k)'}_{n(k)}
        \end{smallmatrix}
        \right)
        (S_{i(k)})
        \Biggr\rangle_{\eta_{j}}\\
        &\quad+
        \sum_{\ell=1}^{k}
        \delta_{j,i(\ell)}
        \sum_{m=2}^{n(\ell)-1}
        \Biggl\langle
        1\otimes1,
        U_{n(1)}^{\eta_{i(1)}}
        \left(
        \begin{smallmatrix}
            b^{(1)}_{1}\\
            b^{(1)'}_{1}
        \end{smallmatrix}
        ;\cdots;
        \begin{smallmatrix}
            b^{(1)}_{n(1)}\\
            b^{(1)'}_{n(1)}
        \end{smallmatrix}
        \right)
        (S_{i(1)})\\
        &\quad\cdots
        U_{n(\ell-1)}^{\eta_{i(\ell-1)}}
        \left(
        \begin{smallmatrix}
            b^{(\ell-1)}_{1}\\
            b^{(\ell-1)'}_{1}
        \end{smallmatrix}
        ;\cdots;
        \begin{smallmatrix}
            b^{(\ell-1)}_{n(\ell-1)}\\
            b^{(\ell-1)'}_{n(\ell-1)}
        \end{smallmatrix}
        \right)
        (S_{i(\ell-1)})
        U_{m-1}
        \left(
        \begin{smallmatrix}
            b^{(\ell)}_{1}\\
            b^{(\ell)'}_{1}
        \end{smallmatrix}
        ;\cdots;
        \begin{smallmatrix}
            b^{(\ell)}_{m-1}\\
            b^{(\ell)'}_{m-1}b^{(\ell+1)}_{m}
        \end{smallmatrix}
        \right)
        (S_{i(\ell)})\\
        &\quad\quad\quad\otimes
        U_{n(\ell)-m}^{\eta_{i(\ell)}}
        \left(
        \begin{smallmatrix}
            b^{(\ell)'}_{m}b^{(\ell)}_{m+1}\\
            b^{(\ell)'}_{m+1}
        \end{smallmatrix}
        ;\cdots;
        \begin{smallmatrix}
            b^{(\ell)}_{n(\ell)}\\
            b^{(\ell)'}_{n(\ell)}
        \end{smallmatrix}
        \right)
        (S_{i(\ell)})
        U_{n(\ell+1)}^{\eta_{i(\ell+1)}}
        \left(
        \begin{smallmatrix}
            b^{(\ell+1)}_{1}\\
            b^{(\ell+1)'}_{1}
        \end{smallmatrix}
        ;\cdots;
        \begin{smallmatrix}
            b^{(\ell+1)}_{n(\ell+1)}\\
            b^{(\ell+1)}_{n(\ell+1)}
        \end{smallmatrix}
        \right)
        (S_{i(\ell+1)})\\
        &\hspace{8cm}
        \cdots
        U_{n(k)}^{\eta_{i(k)}}
        \left(
        \begin{smallmatrix}
            b^{(k)}_{1}\\
            b^{(k)'}_{1}
        \end{smallmatrix}
        ;\cdots;
        \begin{smallmatrix}
            b^{(k)}_{n(k)}\\
            b^{(k)'}_{n(k)}
        \end{smallmatrix}
        \right)
        (S_{i(k)})
        \Biggr\rangle_{\eta_{j}}\\
        &\quad+
        \sum_{\ell=1}^{k}
        \delta_{j,i(\ell)}
        \Biggl\langle
        1\otimes1,
        U_{n(1)}^{\eta_{i(1)}}
        \left(
        \begin{smallmatrix}
            b^{(1)}_{1}\\
            b^{(1)'}_{1}
        \end{smallmatrix}
        ;\cdots;
        \begin{smallmatrix}
            b^{(1)}_{n(1)}\\
            b^{(1)'}_{n(1)}
        \end{smallmatrix}
        \right)
        (S_{i(1)})\\
        &\quad\quad\cdots
        U_{n(\ell-1)}^{\eta_{i(\ell-1)}}
        \left(
        \begin{smallmatrix}
            b^{(\ell-1)}_{1}\\
            b^{(\ell-1)'}_{1}
        \end{smallmatrix}
        ;\cdots;
        \begin{smallmatrix}
            b^{(\ell-1)}_{n(\ell-1)}\\
            b^{(\ell-1)'}_{n(\ell-1)}
        \end{smallmatrix}
        \right)
        (S_{i(\ell-1)})
        U^{\eta_{i(\ell)}}_{n(\ell)-1}
        \left(
        \begin{smallmatrix}
            b^{(\ell)}_{1}\\
            b^{(\ell)'}_{1}
        \end{smallmatrix}
        ;\cdots;
        \begin{smallmatrix}
            b^{(\ell)}_{n(\ell)-1}\\
            b^{(\ell)'}_{n(\ell)-1}b^{(\ell)}_{n(\ell)}
        \end{smallmatrix}
        \right)
        (S_{i(\ell)})\\
        &\quad\quad\quad\otimes
        b^{(\ell)'}_{n(\ell)}
        U_{n(\ell+1)}^{\eta_{i(\ell+1)}}
        \left(
        \begin{smallmatrix}
            b^{(\ell+1)}_{1}\\
            b^{(\ell+1)'}_{1}
        \end{smallmatrix}
        ;\cdots;
        \begin{smallmatrix}
            b^{(\ell+1)}_{n(\ell+1)}\\
            b^{(\ell+1)'}_{n(\ell+1)}
        \end{smallmatrix}
        \right)
        (S_{i(\ell+1)})
        \cdots
        U_{n(k)}^{\eta_{i(k)}}
        \left(
        \begin{smallmatrix}
            b^{(k)}_{1}\\
            b^{(k)'}_{1}
        \end{smallmatrix}
        ;\cdots;
        \begin{smallmatrix}
            b^{(k)}_{n(k)}\\
            b^{(k)'}_{n(k)}
        \end{smallmatrix}
        \right)
        (S_{i(k)})
        \Biggr\rangle_{\eta_{j}}.
    \end{align*}
    When $k\geq2$ or $n(1)\geq2$, all terms of the (most) right-hand side above are equal to $0$ by Proposition \ref{Prop_Bsemicircular_equivalent}. If $k=1$ and $n(1)=1$, then the right-hand side is equal to
    \begin{align*}
        \,&
        \delta_{j,i(i)}
        \left\langle
        1\otimes1,
        b^{(1)}_{1}\otimes b^{(1)'}_{1}
        \right\rangle_{\eta_{i(1)}}
        =
        \delta_{j,i(1)}
        \tau\left(
        \eta_{i(1)}\left(b^{(1)}_{1}\right)
        b^{(1)'}_{1}
        \right).
    \end{align*}
    Thus, we have obtained that 
    \[
    \left\langle
    1\otimes1,
    \mathrm{ev}_{S}^{\otimes2}
    \left(
    \partial_{j}\left[P(X_{1},\dots,X_{d})\right]
    \right)
    \right\rangle_{\eta_{j}}
    =
    \begin{cases}
        0;\quad \mbox{if }k\geq2\mbox{ or }n(1)\geq2,\\
        \delta_{j,i(1)}\tau\left(\eta_{i(1)}\left(b^{(1)}_{1}\right)b^{(1)'}_{1}\right);\quad \mbox{if }k=1\mbox{ and }n(1)=1.
    \end{cases}
    \]
   Therefore, we have proved that
   \[
    \left\langle
    S_{j},P(S_{1},\dots,S_{d})
    \right\rangle_{\tau}
    =
    \left\langle
    1\otimes1,
    \mathrm{ev}_{S}^{\otimes2}
    \left(
    \partial_{j}\left[P(X_{1},\dots,X_{d})\right]
    \right)
    \right\rangle_{\eta_{j}},
    \]
    where
    \begin{align*}
        \,&P(X_{1},\dots,X_{d})
        =
        U^{\eta_{i(1)}}_{n(1)}
        \left(
        \begin{smallmatrix}
            b^{(1)}_{1}\\
            b^{(1)'}_{1}
        \end{smallmatrix}
        ;\cdots;
        \begin{smallmatrix}
            b^{(1)}_{n(1)}\\
            b^{(1)'}_{n(1)}
        \end{smallmatrix}
        \right)
        (X_{i(1)})
        \cdots
        U_{n(k)}^{\eta_{i(k)}}
        \left(
        \begin{smallmatrix}
            b^{(k)}_{1}\\
            b^{(k)'}_{1}
        \end{smallmatrix}
        ;\cdots;
        \begin{smallmatrix}
            b^{(k)}_{n(k)}\\
            b^{(k)'}_{n(k)}
         \end{smallmatrix}
        \right)
        (X_{i(k)})
    \end{align*}
    for any $k\in\mathbb{N}$, any $b_{1},\dots,b_{k}\in B$, $n\in I(k,\mathbb{N})$ and $i\in \mathrm{Alt}(I(k,d))$.
\end{proof}

\section{The divergence operator with respect to $B$-valued semi-circular element}\label{section_divergence}
Let $(A,B,\tau,E)$ be the tuple introduced in the previous section, that is, $B\subset A$ is a unital inclusion of unital $C^*$-algebras with a faithful tracial state $\tau$ and a $\tau$-preserving conditional expectation $E$, and $\eta_{1},\dots,\eta_{d}$ completely positive maps from $B$ to $B$. Let $(S_{1},\dots,S_{d})$ be a $B$-free $B$-valued semi-circular system associated with $\eta_{1},\dots,\eta_{d}$ in $(A,B,\tau,E)$.

\begin{Def}\label{Def_divergence}
    Let $\partial_{j}^*$ be a linear map from $B_{\langle d\rangle}\otimes B_{\langle d\rangle}$ to $B_{\langle d\rangle}$ defined by
    \[
    \partial_{j}^*
    =
    m_{X_{j}}
    -
    \#_{1,1}\circ\left(\mathrm{id}_{B_{\langle d\rangle}}\otimes(\eta_{j}\circ E\circ\mathrm{ev}_{S})\otimes\mathrm{id}_{B_{\langle d\rangle}}\right)\circ\left(\mathrm{id}_{B_{\langle d\rangle}}\otimes\partial_{j}+\partial_{j}\otimes\mathrm{id}_{B_{\langle d\rangle}}\right),
    \]
    where $m_{X_{j}}:B_{\langle d\rangle}\otimes B_{\langle d\rangle}\to B_{\langle d\rangle}$ is defined by $m_{X_{j}}(\xi_{1}\otimes\xi_{2})=\xi_{1}X_{j}\xi_{2}$ for all $\xi_{1},\xi_{2}\in B_{\langle d\rangle}$ and $\#_{1,1}:B_{\langle d\rangle}\otimes B_{\langle d\rangle}\otimes B_{\langle d\rangle}\to B_{\langle d\rangle}$ given by $\#_{1,1}(\xi_{1}\otimes\xi_{2}\otimes\xi_{3})=\xi_{1}\xi_{2}\xi_{3}$ for all $\xi_{1},\xi_{2},\xi_{3}\in B_{\langle d\rangle}$. We call $\partial_{j}^*$ the \textit{divergence operator} with respect to $B$-valued semi-circular element $S_{j}$.
\end{Def}

The following lemma can be confirmed by direct calculation (see e.g., \cite[Proposition 3.1]{l24}.

\begin{Lem}\label{Lem_formula_divergence}
    For any $\Xi\in B_{\langle d\rangle}\otimes B_{\langle d\rangle}$ and $a\in B_{\langle d\rangle}$, we have
    \begin{gather*}
        \partial_{j}^*[(a\otimes1)\Xi]=a\cdot\partial_{j}^*[\Xi]-\bigl(\partial_{j}[a],\Xi\bigr)_{j}\quad\mbox{and}\quad
        \partial_{j}^*[\Xi(1\otimes a)]=\partial_{j}^*[\Xi]\cdot a-\bigl(\Xi,\partial_{j}[a]\bigr)_{j},
    \end{gather*}
    where $\bigl(\cdot,\cdot\bigr)_{j}$ denotes a $B_{\langle d\rangle}$-valued bilinear map on $B_{\langle d\rangle}\otimes B_{\langle d\rangle}$ such that
    \[
    \bigl(\xi_{1}\otimes\xi_{2},\xi_{3}\otimes\xi_{4}\bigr)_{j}
    =
    \xi_{1}\eta_{j}(E[\mathrm{ev}_{S}(\xi_{2}\xi_{3})])\xi_{4}
    \]
    for any $\xi_{1},\xi_{2},\xi_{3},\xi_{4}\in B_{\langle d\rangle}$.
\end{Lem}

The next proposition says that $\partial^*_{j}\circ\partial_{j}$ is the number operator with respect to the $B$-valued Chebyshev family.

\begin{Prop}\label{Prop_numberoperator}
    Let $(U^{\eta_{j}}_{n})_{n}$ be the $B$-valued Chebyshev family associated with $\eta_{j}$, $j\in[d]$. Then, we have 
    \[
    \left(\partial_{j}^*\circ\partial_{j}\right)
    \left[
    U^{\eta_{j}}_{n}
    \left(

    \right)
    (X_{j})
    \]
    for any $(b_{i},b_{i}')\in B\times B$ ($i\in[n]$).
\end{Prop}
\begin{proof}
    Without loss of generality, we may assume that $j=1$. The case when $n=1$ is clear by definition of the divergence operator. 
    
    Suppose that we have shown the case when $1\leq n\leq N$ for some $N\in\mathbb{N}$. Take arbitrary elements $(b_{i},b_{i}')\in B\times B$ ($i\in[N+1]$). Using the recursion relation
    \begin{align*}
        \,&
        U^{\eta_{1}}_{N+1}
        \left(
        %
        \right)
        (X_{1})
        \right].
    \end{align*}
    By the induction hypothesis, we also have
    \begin{align*}
    \,&(\partial_{1}^*\circ\partial_{1})
    \left[
    U_{N-1}^{\eta_{1}}
    \left(
    %
    \right)
    (X_{1}).
    \end{align*}
    By Lemma \ref{Lem_formula_divergence} and the induction hypothesis, we observe that
    \begin{align*}
        \,&
        \partial_{1}^*
        \left[
        \partial_{1}
        \left[
        U_{1}^{\eta_{1}}
        \left(
        %
    \right)
    (S_{1})
    \right]=0$ for any $n\in\mathbb{N}$ by Proposition \ref{Prop_Bsemicircular_equivalent}, the second and the third terms of the above equality vanish. Therefore, we have
    \begin{align*}
        \,&
        (\partial_{1}^*\circ\partial_{1})
        \left[
        U_{N+1}^{\eta_{1}}
        \left(
        %
        \right)
        (X_{i(k)})
    \end{align*}
    for any $k\in\mathbb{N}$, $n\in I(k,\mathbb{N})$, $i\in\mathrm{Alt}(I(k,d))$ and any $b^{(\ell)}_{m}\in B$.
\end{Cor}
\begin{proof}
    We have known the case when $k=1$ by Proposition \ref{Prop_numberoperator}. In the sequel, we will treat the case when $k\geq2$. Let us set
    \[
    W=
    U^{\eta_{i(1)}}_{n(1)}
    \left(
    %
    \right)
    (X_{i(k)}).
    \]
    Then, we observe that
    \begin{align*}
        \,&
        \partial_{j}
        \left[
        U^{\eta_{i(1)}}_{n(1)}
        \left(
        %
        \right)
        (X_{i(k)})
        \right),
    \end{align*}
    and, by Lemma \ref{Lem_formula_divergence},
    \begin{align*}
        \,&
        \left(
        \partial^*_{j}
        \circ
        \partial_{j}
        \right)
        \left[
        U^{\eta_{i(1)}}_{n(1)}
        \left(
        %
        \right)
        (X_{i(k)})
        \right]
        \Biggr)_{j}.
    \end{align*}
    If $k=2$, then the second and the third sums in the above right-hand side are equal to $0$, since $i\in\mathrm{Alt}(I(2,d))$ (that is, $i(\ell-1)\not=i(\ell)(=j)$) and $B\langle X_{i}\rangle\subset\ker(\partial_{j})$ for each $i\in[d]\setminus\{j\}$. Even if $k>2$, then they are also equal to $0$ by Proposition \ref{Prop_Bsemicircular_equivalent} and the definition of the bilinear map $(\cdot,\cdot)_{j}$. By Proposition \ref{Prop_BChebyshev_tensor}, the first term in the above right-hand side is equal to
    \[
    \left(
    \sum_{i(\ell)=j}
    n(\ell)
    \right)
    \cdot
    U^{\eta_{i(1)}}_{n(1)}
    \left(
    \begin{smallmatrix}
        b^{(1)}_{1}\\
        b^{(1)'}_{1}
    \end{smallmatrix}
    ;\cdots;
    \begin{smallmatrix}
        b^{(1)}_{n(1)}\\
        b^{(1)'}_{n(1)}
    \end{smallmatrix}
    \right)
    (X_{i(1)})
    \cdots
    U^{\eta_{i(k)}}_{n(k)}
    \left(
    \begin{smallmatrix}
        b^{(k)}_{1}\\
        b^{(k)'}_{1}
    \end{smallmatrix}
    ;\cdots;
    \begin{smallmatrix}
        b^{(k)}_{n(k)}\\
        b^{(k)'}_{n(k)}
    \end{smallmatrix}
    \right)
    (X_{i(k)}),
    \]
    and hence the proof has been completed.
\end{proof}

\begin{Cor}\label{Cor_directdecomposition}
    Let $X$ be a singleton of a formal variable (i.e., the case of $d=1$). Let $B$ be a unital $C^*$-algebra and $\eta$ a completely positive map on $B$. Then, we have the following direct sum decomposition:
    \[
    B\langle X\rangle
    =
    B\oplus\bigoplus_{n\geq1}^{\mathrm{alg}}\mathrm{span}\{\mathrm{ran}(U^{\eta}_{n})\},
    \] 
    where
    \[
    \mathrm{ran}(U^{\eta}_{n})
    =
    \left\{
    U^{\eta}_{n}
    \left(
    \begin{smallmatrix}
        b_{1}\\
        b_{1}'
    \end{smallmatrix}
    ;
    \begin{smallmatrix}
        b_{2}\\
        b_{2}'
    \end{smallmatrix}
    ;\cdots;
    \begin{smallmatrix}
        b_{n}\\
        b_{n}'
    \end{smallmatrix}
    \right)
    \,\middle|\,
    b_{1},\dots,b_{n},b_{1}',\dots,b_{n}'\in B
    \right\}.
    \]
\end{Cor}
\begin{proof}
    Let $P(X)\in B\langle X\rangle$ satisfy that $P(X)\in\mathrm{span}\{\mathrm{ran}(U^{\eta}_{n})\}\cap\mathrm{span}\{\mathrm{ran}(U^{\eta}_{m})\}$ for some $n,m\in\mathbb{N}$ with $n\not=m$. By Proposition \ref{Prop_numberoperator}, we have $n\cdot P(X)=m\cdot P(X)$, and thus $(m-n)\cdot P(X)=0$. Since $n\not=m$, we have $P(X)=0$.
\end{proof}

\begin{Prop}\label{Prop_divergence_product}
    Let $j\in[d]$. Assume that 
    \[
    \tau\left(\eta_{j}(b_{1})b_{2}\right)=\tau\left(b_{1}\eta_{j}(b_{2})\right)
    \]
    for any $b_{1},b_{2}\in B$
    and that, with an element $\Xi=\sum_{i=1}^{N}\xi_{1,i}\otimes\xi_{2,i}\in B_{\langle d\rangle}\otimes B_{\langle d\rangle}$, we have 
    \begin{align*}
        \left
        \langle
        \mathrm{ev}_{S}(\partial_{j}^*[\Xi]),
        \mathrm{ev}_{S}(\xi)
        \right\rangle_{\tau}
        =
        \left\langle
        \mathrm{ev}_{S}^{\otimes2}(\Xi),
        \mathrm{ev}_{S}^{\otimes 2}
        (\partial_{j}[\xi])
        \right\rangle_{\eta_{j}}
    \end{align*}
    for any $\xi\in B_{\langle d\rangle}$. Then, for any $a\in B_{\langle d\rangle}$, we have
    \begin{align*}
        \left\langle
        \mathrm{ev}_{S}\bigl(\partial_{j}^*\left[(a\otimes1)\Xi\right]\bigr),
        \mathrm{ev}_{S}(\xi)
        \right\rangle_{\tau}
        =
        \left\langle
        \mathrm{ev}_{S}^{\otimes2}\left((a\otimes 1)\Xi\right),
        \mathrm{ev}_{S}^{\otimes2}(\partial_{j}[\xi])
        \right\rangle_{\eta_{j}}
    \end{align*}
    and 
    \begin{align*}
        \left\langle
        \mathrm{ev}_{S}\bigl(\partial_{j}^*\left[\Xi(1\otimes a)\right]\bigr),
        \mathrm{ev}_{S}(\xi)
        \right\rangle_{\tau}
        =
        \left\langle
        \mathrm{ev}_{S}^{\otimes2}
        (\Xi(1\otimes a)),
        \mathrm{ev}_{S}^{\otimes2}(\partial_{j}[\xi])
        \right\rangle_{\eta_{j}}
    \end{align*}
    for any $\xi\in B_{\langle d\rangle}$.
\end{Prop}
\begin{proof}
    Without loss of generality, we may assume $j=1$. It suffices to consider the case when $a$ and $\xi$ are monomials in $B_{\langle d \rangle}$. Then, we observe that
    \begin{align*}
        \,&\left\langle
        \mathrm{ev}_{S}^{\otimes2}((a\otimes 1)\Xi),
        \mathrm{ev}_{S}^{\otimes2}(\partial_{1}[\xi])
        \right\rangle_{\eta_{1}}\\
        &=
        \sum_{\xi=\xi(1)X_{1}\xi(2)}
        \sum_{i=1}^{N}
        \left\langle
        (\mathrm{ev}_{S}(a\xi_{i,1})\otimes\mathrm{ev}_{S}(\xi_{i,2}),
        \mathrm{ev}_{S}(\xi(1))\otimes \mathrm{ev}_{S}(\xi(2))
        \right\rangle_{\eta_{1}}\\
        &=
        \sum_{\xi=\xi(1)X_{1}\xi(2)}
        \sum_{i=1}^{N}
        \tau\bigl(
        \mathrm{ev}_{S}(\xi_{i,2})^*\eta_{1}\left(E[\mathrm{ev}_{S}(\xi_{i,1})^*\mathrm{ev}_{S}(a)^*\mathrm{ev}_{S}(\xi(1))]\right)
        \mathrm{ev}_{S}(\xi(2))
        \bigr)\\
        &=
        \left\langle
        \mathrm{ev}_{S}^{\otimes2}(\Xi),
        \mathrm{ev}_{S}^{\otimes2}\left((a^*\otimes1)\partial_{1}[\xi]\right)
        \right\rangle_{\eta_{1}}\\
        &=
        \left\langle
        \mathrm{ev}_{S}^{\otimes2}(\Xi),
        \mathrm{ev}_{S}^{\otimes2}\left((\partial_{1}[a^*\xi]\right)
        \right\rangle_{\eta_{1}}
        -
        \left\langle
        \mathrm{ev}_{S}^{\otimes2}(\Xi),
        \mathrm{ev}_{S}^{\otimes2}\bigl((\partial_{1}[a^*](1\otimes\xi))\bigr)
        \right\rangle_{\eta_{1}}\\
        &=
        \left\langle
        \mathrm{ev}_{S}(\partial_{1}^*[\Xi]),\mathrm{ev}_{S}(a^*\xi) 
        \right\rangle_{\tau}
        -
        \sum_{a=a_{1}X_{1}a_{2}}
        \left\langle
        \mathrm{ev}_{S}(\xi_{i,1})\otimes\mathrm{ev}_{S}(\xi_{i,2}),
        \mathrm{ev}_{S}(a_{2}^*)\otimes \mathrm{ev}_{S}(a_{1}^*\xi)
        \right\rangle_{\eta_{1}}\\
        &=
        \left\langle
        \mathrm{ev}_{S}(a)\mathrm{ev}_{S}(\partial_{1}^*[\Xi]),\mathrm{ev}_{S}(\xi) 
        \right\rangle_{\tau}
        -
        \sum_{a=a_{1}X_{1}a_{2}}
        \tau\bigl(
        \mathrm{ev}_{S}(\xi_{i,2}^*)\eta_{1}\left(E[\mathrm{ev}_{S}(\xi_{i,1}^*)\mathrm{ev}_{S}(a^*_{2})]\right)\mathrm{ev}_{S}(a^*_{1})\mathrm{ev}_{S}(\xi)
        \bigr)\\
        &=
        \left\langle
        \mathrm{ev}_{S}\left(a\cdot\partial_{1}^*[\Xi]\right),\mathrm{ev}_{S}(\xi) 
        \right\rangle_{\tau}
        -
        \left\langle
        \sum_{a=a_{1}Xa_{2}}
        \mathrm{ev}_{S}(a_{1})\eta_{1}\left(E[\mathrm{ev}_{S}(a_{2})\mathrm{ev}_{S}(\xi_{i,1})]\right)\mathrm{ev}_{S}(\xi_{i,2}),\,
        \mathrm{ev}_{S}(\xi)
        \right\rangle_{\tau}\\
        &=
        \left\langle
        \mathrm{ev}_{S}\left(a\cdot\partial_{1}^*[\Xi]\right) 
        -
        \mathrm{ev}_{S}\left(\bigl(\partial_{1}[a],\Xi\bigr)_{1}\right),
        \mathrm{ev}_{S}(\xi)
        \right\rangle_{\tau}\\
        &=
        \left\langle
        \mathrm{ev}_{S}\bigl(\partial_{1}^*[(a\otimes1)\Xi]\bigr),
        \mathrm{ev}_{S}(\xi)
        \right\rangle_{\tau}
    \end{align*}
    as desired, where the assumption for $\Xi$ was used in the fifth equality. Note that, in this observation, we did not use the assumption of a certain tracial property of $\eta_{1}$.
    Similarly, we also have
    \begin{align*}
        \,&\left\langle
        \mathrm{ev}_{S}^{\otimes2}\left(\Xi(1\otimes a)\right),
        \mathrm{ev}_{S}^{\otimes2}\left(\partial_{1}[\xi]\right)
        \right\rangle_{\eta_{1}}\\
        &=
        \sum_{\xi=\xi(1)X_{1}\xi(2)}\sum_{i=1}^{N}
        \tau\bigl(
        \mathrm{ev}_{S}(a^*)\mathrm{ev}_{S}(\xi_{i,2}^*)\eta_{1}\left(E[\mathrm{ev}_{S}(\xi_{i,1}^*)\mathrm{ev}_{S}(\xi(1))]\mathrm{ev}_{S}(\xi(2))\right)
        \bigr)\\
        &=
        \sum_{\xi=\xi(1)X_{1}\xi(2)}\sum_{i=1}^{N}
        \tau\bigl(
        \mathrm{ev}_{S}(\xi_{i,2}^*)\eta_{1}\left(E[\mathrm{ev}_{S}(\xi_{i,1}^*)\mathrm{ev}_{S}(\xi(1))]\mathrm{ev}_{S}(\xi(2))\mathrm{ev}_{S}(a^*)\right)
        \bigr)\\
        &=
        \left\langle
        \mathrm{ev}_{S}^{\otimes2}(\Xi),
        \mathrm{ev}_{S}^{\otimes2}\left(\partial_{1}[\xi](1\otimes a^*)\right)
        \right\rangle_{\eta_{1}}\\
        &=
        \left\langle
        \mathrm{ev}_{S}^{\otimes2}(\Xi),
        \mathrm{ev}_{S}^{\otimes2}\left(\partial_{1}[\xi a^*]\right)
        \right\rangle_{\eta_{1}}
        -\left\langle
        \mathrm{ev}_{S}^{\otimes2}(\Xi),
        \mathrm{ev}_{S}^{\otimes2}\left((\xi\otimes1)\partial_{1}[a^*]\right)
        \right\rangle_{\eta_{1}}\\
        &=
        \left\langle
        \mathrm{ev}_{S}\left(\partial_{1}^*[\Xi]\right),\mathrm{ev}_{S}(\xi)\mathrm{ev}_{S}(a^*)
        \right\rangle_{\tau}
        -
        \sum_{a=a_{1}X_{1}a_{2}}
        \left\langle
        \mathrm{ev}_{S}^{\otimes2}(\Xi),
        \mathrm{ev}_{S}(\xi a_{2}^*)\otimes \mathrm{ev}_{S}(a^*_{1})
        \right\rangle_{\eta_{1}}\\
        &=
        \left\langle
        \mathrm{ev}_{S}(\partial_{1}^*[\Xi]) \mathrm{ev}_{S}(a),\mathrm{ev}_{S}(\xi)
        \right\rangle_{\tau}
        -
        \sum_{a=a_{1}X_{1}a_{2}}
        \sum_{i=1}^{N}
        \tau\bigl(
        \mathrm{ev}_{S}(\xi_{i,2}^*)\eta_{1}\left(E[\mathrm{ev}_{S}(\xi_{i,1}^*)\mathrm{ev}_{S}(\xi)\mathrm{ev}_{S}(a^*_{2})]\right)\mathrm{ev}_{S}(a_{1}^*)
        \bigr)\\
        &=
        \left\langle
        \mathrm{ev}_{S}(\partial_{1}^*[\Xi]\cdot a),
        \mathrm{ev}_{S}(\xi)
        \right\rangle_{\tau}
        -
        \sum_{a=a_{1}X_{1}a_{2}}
        \sum_{i=1}^{N}
        \tau\bigl(
        E[\mathrm{ev}_{S}(a_{1}^*)\mathrm{ev}_{S}(\xi_{i,2}^*)]\eta_{1}\left(E[\mathrm{ev}_{S}(\xi_{i,1}^*)\mathrm{ev}_{S}(\xi)\mathrm{ev}_{S}(a^*_{2})]\right)
        \bigr)\\
        &=
        \left\langle
        \mathrm{ev}_{S}(\partial_{1}^*[\Xi]\cdot a),\mathrm{ev}_{S}(\xi)
        \right\rangle_{\tau}
        -
        \sum_{a=a_{1}X_{1}a_{2}}
        \sum_{i=1}^{N}
        \tau\bigl(
        \eta_{1}(E[\mathrm{ev}_{S}(a_{1}^*)\mathrm{ev}_{S}(\xi_{i,2}^*)])E[\mathrm{ev}_{S}(\xi_{i,1}^*)\mathrm{ev}_{S}(\xi)\mathrm{ev}_{S}(a^*_{2})]
        \bigr)\\
        &=
        \left\langle
        \mathrm{ev}_{S}(\partial_{1}^*[\Xi]\cdot a),\mathrm{ev}_{S}(\xi)
        \right\rangle_{\tau}
        -
        \sum_{a=a_{1}X_{1}a_{2}}
        \sum_{i=1}^{N}
        \tau\bigl(
        \mathrm{ev}_{S}(a^*_{2})
        \eta_{{\color{red}1}}(E[\mathrm{ev}_{S}(a_{1}^*)\mathrm{ev}_{S}(\xi_{i,2}^*)])\mathrm{ev}_{S}(\xi_{i,1}^*)\mathrm{ev}_{S}(\xi)
        \bigr)\\
        &=
        \left\langle
        \mathrm{ev}_{S}(\partial_{1}^*[\Xi]\cdot a),\mathrm{ev}_{S}(\xi)
        \right\rangle_{\tau}
        -
        \left\langle
        \sum_{a=a_{1}X_{1}a_{2}}
        \sum_{i=1}^{N}
        \mathrm{ev}_{S}(\xi_{i,1})\eta_{1}\left(E[\mathrm{ev}_{S}(\xi_{i,2})\mathrm{ev}_{S}(a_{1})]\right)\mathrm{ev}_{S}(a_{2}),\,
        \mathrm{ev}_{S}(\xi)
        \right\rangle_{\tau}\\
        &=
        \left\langle
        \mathrm{ev}_{S}(\partial_{1}^*[\Xi]\cdot a)
        -
        \mathrm{ev}_{S}
        \left(
        \bigl(
        \Xi,
        \partial_{1}[a]
        \bigr)_{1}
        \right)
        ,
        \mathrm{ev}_{S}(\xi)
        \right\rangle_{\tau}\\
        &=
        \left\langle
        \mathrm{ev}_{S}\left(\partial^*_{1}[\Xi]\cdot a\right),\mathrm{ev}_{S}(\xi)
        \right\rangle_{\tau}
    \end{align*}
    as desired, where the trace property of $\tau$ was used in the second, the seventh and the ninth equalities, the assumption for $\Xi$ was used in the fifth equality, the $\tau$-preserving property of $E$ were used in the seventh and the ninth equalities and the assumption for $\eta_{1}$ is used in the eighth equality.
\end{proof}

\begin{Cor}\label{Cor_integralbyparts}
    Let $j\in[d]$. Assume that
    \[
    \tau\left(\eta_{j}(b_{1})b_{2}\right)=\tau\left(b_{1}\eta_{j}(b_{2})\right)
    \]
    for any $b_{1},b_{2}\in B$. Then, we have 
    \[
    \left\langle
    \mathrm{ev}_{S}\left(\partial_{j}^*[\Xi]\right),
    \mathrm{ev}_{S}(\xi)
    \right\rangle_{\tau}
    =
    \left\langle
    \mathrm{ev}_{S}^{\otimes2}(\Xi),
    \mathrm{ev}_{S}^{\otimes2}\left(\partial_{j}[\xi]\right)
    \right\rangle_{\eta_{j}}
    \]
    for any $\xi\in B_{\langle d\rangle}$ and $\Xi\in B_{\langle d\rangle}\otimes B_{\langle d\rangle}$.
\end{Cor}
\begin{proof}
    With the divergence operator $\partial_{j}^*$, the formula of Proposition \ref{Prop_Bsemicircular_Stein} can be written as follows.
    \[
    \left\langle
    \mathrm{ev}_{S}(\partial_{j}^*[1\otimes1]),\mathrm{ev}_{S}(\xi)
    \right\rangle_{\tau}
    =
    \left\langle
    1\otimes1,
    \mathrm{ev}_{S}^{\otimes2}(\partial_{j}[\xi])
    \right\rangle_{\eta_{j}},\quad 
    \xi\in B_{\langle d\rangle}.
    \]
    Choose an arbitrary $\Xi\in B_{\langle d\rangle}\otimes B_{\langle d\rangle}$ and fix an  arbitrary expression $\sum_{i=1}^{N}\xi_{i,1}\otimes\xi_{i,2}$ with $\xi_{i,1},\xi_{i,2}\in B_{\langle d\rangle}$. By Propositions \ref{Prop_Bsemicircular_Stein} and \ref{Prop_divergence_product}, we have, for any $i\in[N]$,
    \begin{align*}
        \,&
        \left\langle
        \mathrm{ev}_{S}^{\otimes2}(1\otimes\xi_{i,2}),
        \mathrm{ev}_{S}^{\otimes2}(\partial_{j}[\xi])
        \right\rangle_{\eta_{j}}
        =
        \left\langle
        \mathrm{ev}_{S}\left(\partial_{j}^*[1\otimes\xi_{i,2}]\right),
        \mathrm{ev}_{S}(\xi)
        \right\rangle_{\tau},\quad\xi\in B_{\langle d\rangle}.
    \end{align*}
    Moreover, using Proposition \ref{Prop_divergence_product} again, we have
    \begin{align*}
        \,&\left\langle
        \mathrm{ev}_{S}^{\otimes2}(\xi_{i,1}\otimes\xi_{i,2}),
        \mathrm{ev}_{S}^{\otimes2}
        (\partial_{j}[\xi])
        \right\rangle_{\eta_{j}}
        =
        \left\langle
        \mathrm{ev}_{S}\left(\partial_{j}^*[\xi_{i,1}\otimes\xi_{i,2}]\right),
        \mathrm{ev}_{S}(\xi)
        \right\rangle_{\tau}.
    \end{align*}
    By linearity, we obtain that
    \[
    \left\langle
    \mathrm{ev}_{S}\left(\partial_{j}^*[\Xi]\right),
    \mathrm{ev}_{S}(\xi)
    \right\rangle_{\tau}
    =
    \left\langle
    \mathrm{ev}_{S}^{\otimes2}(\Xi),
    \mathrm{ev}_{S}^{\otimes2}\left(\partial_{j}[\xi]\right)
    \right\rangle_{\eta_{j}}
    \]
    for any $\xi\in B_{\langle d\rangle}$ and $\Xi\in B_{\langle d\rangle}\otimes B_{\langle d\rangle}$.
\end{proof}

\section{A characterization of $B$-valued semi-circular element by Poincar\'{e} type inequality}\label{section_characterization}
\subsection{Matrix amplification for $B$-valued Chebyshev family}
In this subsection we prepare a key lemma, which will be used to prove our main result. Here, we sometimes use the following notation:
\[
\bigoplus_{1\leq j\leq k}^{\to}v_{j}=

\]
for some vectors $v_{1},\dots,v_{k}$.

\begin{Lem}\label{lem_matampChebyshev1}
    Let $B$ be a unital algebra over $\mathbb{C}$ and $\eta$ a linear map from $B$ to itself. Then, we have
    \begin{align*}
        \,&
        U^{\eta\otimes\mathrm{id}_{\ell}}_{n}
        \left(
        %
    \end{align*}
    for any $b_{i}(j),b_{i}'(j)\in B$ with $i\in[n]$ and $j\in[k]$ and any $\ell\geq k$.
\end{Lem}
\begin{proof}
    Take arbitrary $\ell,k\in\mathbb{N}$ with $\ell\geq k$. The case of $n=1,2$ can easily be treated. Assume that we have shown the desired formula in the case of $n\leq N$ for some $N\geq2$. Then, for any $b_{i,j},b_{i,j}'\in B$ with $i\in[k]$ and $j\in[N+1]$, we observe that
    \begin{align*}
        \,&
        U^{\eta\otimes\mathrm{id}_{\ell}}_{N+1}
        \left(
        %
.
    \end{align*}
    Thus, we obtain the desired assertion by induction.
\end{proof}

\begin{Lem}\label{lem_matampChebyshev2}
    Let $B$ be a unital algebra over $\mathbb{C}$ and $\eta_{1},\dots,\eta_{d}$ linear maps from $B$ to itself. Then, we have
    \begin{align*}
        \,&
        %
    \end{align*}
    for any $b_{t,f}(j),b_{t,f}'(j)\in B$, any $s,k,\ell\in\mathbb{N}$ with $\ell\geq k$ and $i\in\mathrm{Alt}(I(s,d))$, where we write
    \[
    \mathbf{b}_{t,f}
    =
    %
.
    \end{align*}
    Using Lemma \ref{lem_matampChebyshev1}, the final right-hand side is equal to
    \begin{align*}
        \,&
        %
.
    \end{align*}
    Thus, we are done.
\end{proof}

\subsection{A characterization of $B$-valued semi-circular element by Poincar\'{e} type inequality}
Let $B\subset A$ be a unital inclusion of unital $C^*$-algebras with a faithful tracial state $\tau$ on $A$ and a $\tau$-preserving conditional expectation $E:A\to B$. Let also $\eta_{1},\dots,\eta_{d}$ be completely positive maps from $B$ to $B$ such that each $\eta_{j}$, $j\in[d]$, satisfies $\tau(\eta_{j}(b)b')=\tau(b\eta_{j}(b'))$ for any $b,b'\in B$. We set $\eta=(\eta_{1},\dots,\eta_{d})$.

We denote by $\mathcal{T}_{\eta}B_{\langle d\rangle}$ the following set:
\begin{align*}
    \,&
    \Biggl\{
    b+
    \sum_{1\leq k\leq N}
    \sum_{\substack{1\leq \ell\leq k\\i\in\mathrm{Alt}(I(\ell,d))}}
    \sum_{\substack{n\in I(\ell,\mathbb{N})\\n(1)+\cdots+n(\ell)=k}}\\
    &\hspace{1cm}
    U_{n(1)}^{\eta_{i(1)}}
    \left(
    \begin{smallmatrix}
        b^{(k,\ell,1;i)}_{1}\\
        b^{(k,\ell,1;i)'}_{1}
    \end{smallmatrix}
    ;\cdots;
    \begin{smallmatrix}
        b^{(k,\ell,1;i)}_{n(1)}\\
        b^{(k,\ell,1;i)'}_{n(1)}
    \end{smallmatrix}
    \right)
    (X_{i(1)})
    \cdots
    U_{n(\ell)}^{\eta_{i(\ell)}}
    \left(
    \begin{smallmatrix}
        b^{(k,\ell,\ell;i)}_{1}\\
        b^{(k,\ell,\ell;i)'}_{1}
    \end{smallmatrix}
    ;\cdots;
    \begin{smallmatrix}
        b^{(k,\ell,\ell;i)}_{n(\ell)}\\
        b^{(k,\ell,\ell;i)'}_{n(\ell)}
    \end{smallmatrix}
    \right)
    (X_{i(\ell)})\\
    &\hspace{9cm}
    \Biggl|\,
    N\in\mathbb{N},b,b^{(k,\ell,t;i)}_{s},b^{(k,\ell,t;i)'}_{s}\in B
    \Biggr\}.
\end{align*}

\begin{Rem}\label{Rem_expression}
    The set $\mathcal{T}_{\eta}B_{\langle d\rangle}$ is not, in general, a linear space (see Remark \ref{Rem_exmaple}). If we call
    \[
    U_{n(1)}^{\eta_{i(1)}}
    \left(
    \begin{smallmatrix}
        \cdot\\
        \cdot
    \end{smallmatrix}
    ;\cdots;
    \begin{smallmatrix}
        \cdot\\
        \cdot
    \end{smallmatrix}
    \right)
    (X_{i(1)})
    \cdots
    U_{n(\ell)}^{\eta_{i(\ell)}}
    \left(
    \begin{smallmatrix}
        \cdot\\
        \cdot
    \end{smallmatrix}
    ;\cdots;
    \begin{smallmatrix}
        \cdot\\
        \cdot
    \end{smallmatrix}
    \right)
    (X_{i(\ell)})
    \]
    a $(k,\ell,n,i)$-term ($1\leq\ell\leq k$, $i\in \mathrm{Alt}(I(\ell,d))$ and $n\in I(\ell,\mathbb{N})$), then
    the above definition of $\mathcal{T}B_{\langle d\rangle}$ means that \emph{every element in $\mathcal{T}B_{\langle d\rangle}$ admits an expression such that each $(k,\ell,n,i)$-term ($1\leq\ell\leq k$, $i\in \mathrm{Alt}(I(\ell,d))$ and $n\in I(\ell,\mathbb{N})$) can appear at most only once}. For example,
    \[
    \mathcal{T}_{\eta}B_{\langle 1\rangle}
    =
    \left\{
    b+
    \sum_{1\leq k\leq N}
    U_{k}^{\eta}
    \left(
    \begin{smallmatrix}
        b^{(k)}_{1}\\
        b^{(k)'}_{1}
    \end{smallmatrix}
    ;\cdots;
    \begin{smallmatrix}
        b^{(k)}_{k}\\
        b^{(k)'}_{k}
    \end{smallmatrix}
    \right)
    (X)
    \,\middle|\,
    N\in\mathbb{N},b,b^{(j)}_{s},b^{(j)'}_{s}\in B
    \right\}
    \]
    in the case of $d=1$.
\end{Rem}

\begin{Rem}\label{Rem_exmaple}
    When $B=\mathbb{C}$, we have $\mathcal{T}\mathbb{C}_{\langle d\rangle}=\mathbb{C}_{\langle d\rangle}$. However, $\mathcal{T}B_{\langle d\rangle}\not=B_{\langle d\rangle}$ in general as follows.
    Let $c_{0}(\mathbb{N})$ be the non-unital $C^*$-algebra of all sequences $x=(x_{n})_{n=1}^{\infty}$ of complex number such that $|x_{n}|\to0$ as $n\to\infty$ and $c_{0}(\mathbb{N})^{\sim}$ the unitalization of $c_{0}(\mathbb{N})$. Define $P_{n}(X)\in \left(c_{0}(\mathbb{N})^{\sim}\right)\langle X\rangle$ to be
    $
    P_{n}(X)=\sum_{k=1}^{n}k\delta_{k}X\delta_{k}
    $,
    where $\delta_{k}=(\delta_{k,n})_{n=1}^{\infty}$.
    Then, we have $P_{n}(X)\not\in\mathcal{T}\left(c_{0}(\mathbb{N})^{\sim}\right)\langle X\rangle$ for any $n\in\mathbb{N}$. Indeed, if we had $P_{n}(X)\in\mathcal{T}\left(c_{0}(\mathbb{N})^{\sim}\right)\langle X\rangle$, then there would exist elements $a=(a_{k})_{k=1}^{\infty},a'\in c_{0}(\mathbb{N})^{\sim}$ such that $P_{n}(X)=aXa'$ by Corollary \ref{Cor_directdecomposition}.
    Let $b,c,b',c'\in c_{0}(\mathbb{N})^{\sim}$ be $a=b+c$, $a'=b'+c'$ and $b,b'\in\mathrm{span}\{\delta_{k}\,|\,k\in[n]\}$, $c,c'\in\mathrm{span}\{e_{j}\,|\,j>n\}$. 
    Applying $\partial_{X:c_{0}(\mathbb{N})^{\sim}}$ to both sides of $P_{n}(X)=aXa'$, we have
    $
    \sum_{k=1}^{n}k\delta_{k}\otimes\delta_{k}=a\otimes a',
    $
    and hence, taking multiplication of both sides above, $a_{k}a_{k}'=k$ for any $k\leq n$ and $a_{j}a_{j}'=0$ 
    for any $j>n$, that is, $cc'=0$. 
    Multiplying $1\otimes \delta_{j}$ ($j>n$) and $1\otimes \delta_{1}$ to
    \[
    \sum_{k=1}^{n}k\delta_{k}\otimes\delta_{k}=bXb'+bXc'+cXb'+cXc'
    \]
    from left and right, respectively, we have $0=a_{1}a_{j}' \delta_{1}X\delta_{j}$, that is, $a_{j}'=0$ (since $a_{1}\not=0$). Similarly, we can see $a_{j}=0$ ($j>n$). Hence, we have $a=b$ and $a'=b'$. Thus, we observe that
    \[
    \sum_{\substack{j,k\in[n]\\j\not=k}}a_{j}a_{k}'\delta_{j}X\delta_{k}=P_{n}(X)-bXb'=0.
    \]
    Since $\delta_{j}X\delta_{k}$, $j\not=k\in[n]$, are linearly independent in $\left(c_{0}(\mathbb{N})^{\sim}\right)\langle X\rangle$, it follows that $a_{j}a_{k}'=0$ for any $j,k\in[n]$ with $j\not=k$, that is, $a_{j}=0$ or $a_{k}'=0$. However, this is a contradiction for $a_{k}a_{k}'=k$ ($k\in[n]$).
\end{Rem}

\begin{Rem}\label{rem_bncpspre}
    \begin{enumerate}
    \item Let $(A,B,\tau,E)$ be a tracial $B$-valued $C^*$-probability space. Then, each $(M_{n}(A),M_{n}(B),\tau\otimes\mathrm{tr}_{n},E\otimes\mathrm{id}_{n})$ is a tracial $M_{n}(B)$-valued $C^*$-probability space. Note that $E\otimes\mathrm{id}_{n}$ is $\tau\otimes\mathrm{tr}_{n}$-preserving. Let $\eta$ be a linear map from $B$ to itself. If $\tau(\eta(b)b')=\tau(b\eta(b'))$ for any $b,b'\in B$, then $(\tau\otimes\mathrm{tr}_{n})\left((\eta\otimes\mathrm{id}_{n})(S)T\right)=(\tau\otimes\mathrm{tr}_{n})\left(S(\eta\otimes\mathrm{id}_{n})(T)\right)$ for any $S,T\in M_{n}(B)$.
    \item If $S=(S_{1},\dots,S_{d})$ is a $B$-free $B$-valued semi-circular system with mean $0$ and variance $\eta=(\eta_{1},\dots,\eta_{d})$ in $(A,B,\tau,E)$, then $S\otimes I_{N}=(S_{1}\otimes I_{N},\dots,S_{d}\otimes I_{N})$ is an $M_{N}(B)$-free $M_{N}(B)$-valued semi-circular system with mean $0$ and variance $\eta\otimes I_{N}=(\eta_{1}\otimes I_{N},\dots,\eta_{d}\otimes I_{N})$ in $(M_{N}(A),M_{N}(B),\tau\otimes\mathrm{tr}_{N},E\otimes\mathrm{id}_{N})$.
    \item For any $N\in\mathbb{N}$, there exists a unique $*$-homomorphism $\mathrm{ev}_{X\otimes I_{N}}$ from $M_{N}(B)\langle Y_{1},\dots,Y_{d}\rangle$ to $M_{N}(B\langle X_{1},\dots,X_{d}\rangle)$ such that $b\mapsto b$, $b\in M_{N}(B)$, and $Y_{j}\mapsto X_{j}\otimes I_{N}$, $j\in[d]$ by the universality of the free product of unital algebras.
    \end{enumerate}
\end{Rem}

Lemma \ref{lem_matampChebyshev2} implies the following property of the class $\mathcal{T}_{\eta}B_{\langle d\rangle}$:

\begin{Prop}\label{prop_chebyshevorthog_formla}
    For any $P\in B_{\langle d\rangle}$, there exist $N\in \mathbb{N}$ and $\widetilde{P}\in \mathcal{T}_{\eta\otimes\mathrm{id}_{N}}(M_{N}(B))\langle Y_{1},\dots,Y_{d}\rangle$ such that 
    \begin{enumerate}
        \item $P$ appears as the corner of $\widetilde{P}$, i.e., 
        $\begin{bmatrix}
            P(X_{1},\dots,X_{d})&\\
            &\mbox{\Large0}_{N-1}
        \end{bmatrix}
        =\mathrm{ev}_{X\otimes I_{N}}(\widetilde{P})$,\\
        \item We have $\left\|\mathrm{ev}_{S\otimes I_{N}}^{\otimes2}(\partial_{Y_{j}}\widetilde{P})\right\|^2_{\eta_{j}\otimes\mathrm{id}_{N}}=\frac{1}{N}\left\|\mathrm{ev}_{S}^{\otimes2}(\partial_{j}P)\right\|^2_{\eta_{j}}$,
    \end{enumerate}
    where the symbol $\|\cdot\|_{\eta_{j}\otimes\mathrm{id}_{N}}$ is the norm with respect to the (pre-)inner product $\langle\cdot,\cdot\rangle_{\eta_{j}\otimes\mathrm{id}_{N}}$ constructed in such a way that
    \[\langle T_{1}\otimes T_{2},T_{1}'\otimes T_{2}'\rangle_{\eta_{j}\otimes\mathrm{id}_{N}}=(\tau\otimes\mathrm{tr}_{N})\left(T_{2}^*\left((\eta_{j}\otimes\mathrm{id}_{N})\circ(E\otimes\mathrm{id}_{N})\right)\left(T_{1}^*T_{1}'\right)T_{2}'\right)\]
    for any $T_{1},T_{1}',T_{2},T_{2}'\in M_{N}(A)$.
\end{Prop}
\begin{proof}
    Because $B$ and $\mathrm{ran}(U^{\eta_{i}}_{n}(-)(X_{i}))$ ($n\in\mathbb{N}$, $i\in[d]$) generate $B_{\langle d\rangle}$, any $P\in B_{\langle d\rangle}$ admits the following expression:
    \begin{align*}
        \,&P(X_{1},\dots,X_{d})=b+\sum_{m=1}^{M}\sum_{s=1}^{m}\sum_{\substack{n_{s}\in I(s,\mathbb{N})\\n_{s}(1)+\cdots+n_{s}(s)=m}}\sum_{i_{s}\in\mathrm{Alt}(I(s,d))}\sum_{j=1}^{k(m,s,n_{s},i_{s})}\\
        &\quad
            U_{n_{s}(1)}^{\eta_{i_{s}(1)}}
            \left(
            \begin{matrix}
                b_{i_{s},1,1}(j)\\
                b_{i_{s},1,1}'(j)
            \end{matrix};\cdots;
            \begin{matrix}
                b_{i_{s},1,n_{s}(1)}(j)\\
                b_{i_{s},1,n_{s}(1)}'(j)
            \end{matrix}
            \right)
            (X_{i_{s}(1)})\cdots
            U_{n_{s}(s)}^{\eta_{i_{s}(s)}}
            \left(
            \begin{matrix}
                b_{i_{s},s,1}(j)\\
                b_{i_{s},s,1}'(j)
            \end{matrix};\cdots;
            \begin{matrix}
                b_{i_{s},s,n_{s}(s)}(j)\\
                b_{i_{s},s,n_{s}(s)}'(j)
            \end{matrix}
            \right)
            (X_{i_{s}(s)})
    \end{align*}
    for some $M\in\mathbb{N}$, $b_{i_{s},t,f}(j),b_{i_{s},t,f}'(j)\in B$. 
    Set 
    \begin{align*}
        \,&N=\max\Bigl\{k(m,s,n_{s},i_{s})\,\Bigl|\,m\in[M],s\in[m],\\
        &\hspace{2cm}n_{s}\in I(s,\mathbb{N})\mbox{ with }n_{s}(1)+\cdots+n_{s}(s)=m,i_{s}\in\mathrm{Alt}(I(s,d))\Bigr\},
    \end{align*} 
    \[\mathbf{b}_{i_{s},t,f;n_{s}}=
    \left[
    \begin{smallmatrix}
        b_{i_{s},t,f}(1)&&\\
        &\ddots&\\
        &&b_{i_{s},t,f}(k(m,s,n_{s},i_{s}))
    \end{smallmatrix}
    \right],\quad
    \mathbf{b}_{i_{s},t,f;n_{s}}'=
    \left[
    \begin{smallmatrix}
        b_{i_{s},t,f}'(1)&&\\
        &\ddots&\\
        &&b_{i_{s},t,f}'(k(m,s,n_{s},i_{s}))
    \end{smallmatrix}
    \right]\]
    and 
    \[\widetilde{\mathbf{b}}_{i_{s},t,f;n_{s}}=
    \left[
    \begin{smallmatrix}
        \mathbf{b}_{i_{s},t,f;n_{s}}&\\
        &\mbox{\Large 0}_{N-k(m,s,n_{s},i_{s})}
    \end{smallmatrix}
    \right],\quad
    \widetilde{\mathbf{b}}_{i_{s},t,f;n_{s}}'=
    \left[
    \begin{smallmatrix}
        \mathbf{b}_{i_{s},t,f;n_{s}}'&\\
        &\mbox{\Large 0}_{N-k(m,s,n_{s},i_{s})}
    \end{smallmatrix}
    \right].\]
    Then, we observe that
    \begin{align*}
        \,&
        \begin{bmatrix}
            P(X_{1},\dots,X_{d})&\\
            &\mbox{\Large 0}_{N-1}
        \end{bmatrix}
        =
        \begin{bmatrix}
            b&\\
            &\mbox{\Large 0}_{N-1}
        \end{bmatrix}
        +
        \sum_{m=1}^{M}\sum_{s=1}^{m}\sum_{\substack{n_{s}\in I(s,\mathbb{N})\\n_{s}(1)+\cdots+n_{s}(s)=m}}\sum_{i_{s}\in\mathrm{Alt}(I(s,d))}\\
        &\quad
        \left[\begin{smallmatrix}
            \begin{matrix}
                1&\cdots&1\\
            \end{matrix}\\
            \mbox{\Large 0}_{N-1,N}
        \end{smallmatrix}\right]
        U^{\eta_{i_{s}(1)}\otimes\mathrm{id}_{N}}_{n_{s}(1)}
        \left(
        \begin{matrix}
        \widetilde{\mathbf{b}}_{i_{s},1,1}\\
        \widetilde{\mathbf{b}}'_{i_{s},1,1}
        \end{matrix}
        ;\cdots;
         \begin{matrix}
            \widetilde{\mathbf{b}}_{i_{s},1,n_{s}(1)}\\
            \widetilde{\mathbf{b}}_{i_{s},1,n_{s}(1)}'
        \end{matrix}
        \right)(X_{i_{s}(1)}\otimes I_{N})\cdots\\
        &\hspace{2cm}\cdots
        U^{\eta_{i_{s}(s)}\otimes\mathrm{id}_{N}}_{n_{s}(s)}
        \left(
        \begin{matrix}
        \widetilde{\mathbf{b}}_{i_{s},s,1}\\
        \widetilde{\mathbf{b}}'_{i_{s},s,1}
        \end{matrix}
        ;\cdots;
         \begin{matrix}
            \widetilde{\mathbf{b}}_{i_{s},s,n_{s}(s)}\\
            \widetilde{\mathbf{b}}_{i_{s},s,n_{s}(s)}'
        \end{matrix}
        \right)(X_{i_{s}(s)}\otimes I_{N})
        \left[\begin{smallmatrix}
            \begin{matrix}
                1\\
                \vdots\\
                1
            \end{matrix}&\mbox{\Large 0}_{N,N-1}
        \end{smallmatrix}\right]
    \end{align*}
    by Lemma \ref{lem_matampChebyshev2}.
    Thus, we can choose $\widetilde{P}$ as an element of $\mathcal{T}_{\eta\otimes\mathrm{id}_{N}}(M_{k}(B))_{\langle d\rangle}$ defined by
    \begin{align*}
        \,&\widetilde{P}(Y_{1},\dots,Y_{d})
        =
        \begin{bmatrix}
            b&\\
            &\mbox{\Large 0}_{N-1}
        \end{bmatrix}
        +
        \sum_{m=1}^{M}\sum_{s=1}^{m}\sum_{\substack{n_{s}\in I(s,\mathbb{N})\\n_{s}(1)+\cdots+n_{s}(s)=m}}\sum_{i_{s}\in\mathrm{Alt}(I(s,d))}\\
        &\quad
        \left[\begin{smallmatrix}
            \begin{matrix}
                1&\cdots&1\\
            \end{matrix}\\
            \mbox{\Large 0}_{N-1,N}
        \end{smallmatrix}\right]
        U^{\eta_{i_{s}(1)}\otimes\mathrm{id}_{N}}_{n_{s}(1)}
        \left(
        \begin{matrix}
        \widetilde{\mathbf{b}}_{i_{s},1,1}\\
        \widetilde{\mathbf{b}}'_{i_{s},1,1}
        \end{matrix}
        ;\cdots;
         \begin{matrix}
            \widetilde{\mathbf{b}}_{i_{s},1,n_{s}(1)}\\
            \widetilde{\mathbf{b}}_{i_{s},1,n_{s}(1)}'
        \end{matrix}
        \right)(Y_{i_{s}(1)})\cdots\\
        &\hspace{2cm}\cdots
        U^{\eta_{i_{s}(s)}\otimes\mathrm{id}_{N}}_{n_{s}(s)}
        \left(
        \begin{matrix}
        \widetilde{\mathbf{b}}_{i_{s},s,1}\\
        \widetilde{\mathbf{b}}'_{i_{s},s,1}
        \end{matrix}
        ;\cdots;
         \begin{matrix}
            \widetilde{\mathbf{b}}_{i_{s},s,n_{s}(s)}\\
            \widetilde{\mathbf{b}}_{i_{s},s,n_{s}(s)}'
        \end{matrix}
        \right)(Y_{i_{s}(s)})
        \left[\begin{smallmatrix}
            \begin{matrix}
                1\\
                \vdots\\
                1
            \end{matrix}&\mbox{\Large 0}_{N,N-1}
        \end{smallmatrix}\right].
    \end{align*}
    
    Also, we can see (2) for the above $\widetilde{P}$ by direct computation using Corollary \ref{cor_orthogonality_chebyshev} and Proposition \ref{prop_chebyshevorthog_formla}. 
\end{proof}

The following theorem is a $B$-valued analogue of Biane's theorem \cite[Theorem 5.1]{b03}:

\begin{Thm}\label{Thm_character_Bsemicircular}
    Let $(S_{1},\dots,S_{d})$ be a $B$-free family of self-adjoint $B$-valued non-commutative random variables in $(A,B,\tau,E)$ with mean $0$ and variance $\eta=(\eta_{1},\dots,\eta_{d})$, respectively. Then, the following are equivalent:
    \begin{enumerate}
        \item $(S_{1},\dots,S_{d})$ is a $B$-valued semi-circular system associated with $\eta_{1},\dots,\eta_{d}$.
        \item We have
        \begin{align*}
        \,&
        \left\|
        P(S_{1},\dots,S_{d})
        -
        E\left[
        P(S_{1},\dots,S_{d})
        \right]
        \right\|_{\tau}^2
        \leq
        \sum_{j=1}^{d}
        \left\|
        \mathrm{ev}_{S}^{\otimes2}
        \left(
        \partial_{j}
        \left[
        P(X_{1},\dots,X_{d})
        \right]
        \right)
        \right\|_{\eta_{j}}^2
        \end{align*}
        for any $P(X_{1},\dots,X_{d})\in\mathcal{T}_{\eta}B_{\langle d\rangle}$.
        \item We have
        \[
        \|P(S_{1},\dots,S_{d})-E[P(S_{1},\dots,S_{d})]\|_{\tau}^2\leq\sum_{j=1}^{d}\|\mathrm{ev}_{S}^{\otimes2}(\partial_{j}[P(X_{1},\dots,X_{d})])\|_{\eta_{j}}^2
        \]
        for any $P\in B_{\langle d\rangle}$.
    \end{enumerate}
\end{Thm}
\begin{proof}
    (1)$\Rightarrow$(2): Let $(S_{1},\dots,S_{d})$ be a $B$-free $B$-valued semi-circular system associated with $\eta_{1},\dots,\eta_{d}$. 
    Choose an arbitrary $P(X_{1},\dots,X_{d})\in\mathcal{T}_{\eta}B_{\langle d\rangle}$ and fix an expression of $P(X_{1},\dots,X_{d})$ as follows.
    \begin{align*}
        \,&
        P(X_{1},\dots,X_{d})=
        b+
        \sum_{1\leq k\leq N}
        \sum_{\substack{1\leq \ell\leq k\\i\in\mathrm{Alt}(I(\ell,d))}}
        \sum_{\substack{n\in I(\ell,\mathbb{N})\\n(1)+\cdots+n(\ell)=k}}\\
        &\hspace{1cm}
        U_{n(1)}^{\eta_{i(1)}}
        \left(
        \begin{smallmatrix}
            b^{(k,\ell,1;i)}_{1}\\
            b^{(k,\ell,1;i)'}_{1}
        \end{smallmatrix}
        ;\cdots;
        \begin{smallmatrix}
            b^{(k,\ell,1;i)}_{n(1)}\\
            b^{(k,\ell,1;i)'}_{n(1)}
        \end{smallmatrix}
        \right)
        (X_{i(1)})
        \cdots
        U_{n(\ell)}^{\eta_{i(\ell)}}
        \left(
        \begin{smallmatrix}
            b^{(k,\ell,\ell;i)}_{1}\\
            b^{(k,\ell,\ell;i)'}_{1}
        \end{smallmatrix}
        ;\cdots;
        \begin{smallmatrix}
            b^{(k,\ell,\ell;i)}_{n(\ell)}\\
            b^{(k,\ell,\ell;i)'}_{n(\ell)}
        \end{smallmatrix}
        \right)
        (X_{i(\ell)}).
    \end{align*}
    Moreover, we may assume that $E[P(S_{1},\dots,S_{d})]=0$, that is, $b=0$ without loss of generality. Then, by Corollaries \ref{Cor_integralbyparts} and \ref{Cor_numberoperator}, we have
    \begin{align*}
        \,&
        \sum_{j=1}^{d}
        \left\|
        \mathrm{ev}_{S}^{\otimes2}
        \left(
        \partial_{j}\left[P(X_{1},\dots,X_{d})\right]
        \right)
        \right\|_{\eta_{j}}^2
        =
        \sum_{j=1}^{d}
        \Bigl\langle
        \mathrm{ev}_{S}
        \left(
        \partial_{j}^*
        \left(
        \partial_{j}
        [P(X_{1},\dots,X_{d})]\right)
        \right),
        P(S_{1},\dots,S_{d})
        \Bigr\rangle_{\tau}\\
        &=
        \sum_{j=1}^{d}
        \sum_{k=1}^{N}
        \sum_{\substack{1\leq \ell\leq k\\i\in\mathrm{Alt}(I(\ell,d))}}
        \sum_{\substack{n\in I(\ell,\mathbb{N})\\n(1)+\cdots+n(\ell)=k}}\\
        &\quad
        \Biggl\langle
        \mathrm{ev}_{S}
        \Biggl(
        \partial_{j}^*
        \Biggl[
        \partial_{j}
        \Biggl[
        U_{n(1)}^{\eta_{i(1)}}
        \left(
        \begin{smallmatrix}
            b^{(k,\ell,1;i)}_{1}\\
            b^{(k,\ell,1;i)'}_{1}
        \end{smallmatrix}
        ;\cdots;
        \begin{smallmatrix}
            b^{(k,\ell,1;i)}_{n(1)}\\
            b^{(k,\ell,1;i)'}_{n(1)}
        \end{smallmatrix}
        \right)
        (X_{i(1)})\\
        &\hspace{4cm}
        \cdots
        U_{n(\ell)}^{\eta_{i(\ell)}}
        \left(
        \begin{smallmatrix}
            b^{(k,\ell,\ell;i)}_{1}\\
            b^{(k,\ell,\ell;i)'}_{1}
        \end{smallmatrix}
        ;\cdots;
        \begin{smallmatrix}
            b^{(k,\ell,\ell;i)}_{n(\ell)}\\
            b^{(k,\ell,\ell;i)'}_{n(\ell)}
        \end{smallmatrix}
        \right)
        (X_{i(\ell)})
        \Biggr]
        \Biggr]
        \Biggr),
        P(S_{1},\dots,S_{d})
        \Biggr\rangle_{\tau}\\
        &
        =
        \sum_{k=1}^{N}
        \sum_{\substack{1\leq \ell\leq k\\i\in\mathrm{Alt}(I(\ell,d))}}
        \sum_{\substack{n\in I(\ell,\mathbb{N})\\n(1)+\cdots+n(\ell)=k}}
        \left(\sum_{j=1}^{d}
        \sum_{i(t)=j}n(t)\right)\\
        &\quad\Biggl\langle
        U_{n(1)}^{\eta_{i(1)}}
        \left(
        \begin{smallmatrix}
            b^{(k,\ell,1;i)}_{1}\\
            b^{(k,\ell,1;i)'}_{1}
        \end{smallmatrix}
        ;\cdots;
        \begin{smallmatrix}
            b^{(k,\ell,1;i)}_{n(1)}\\
            b^{(k,\ell,1;i)'}_{n(1)}
        \end{smallmatrix}
        \right)
        (S_{i(1)})\\
        &\hspace{4cm}
        \cdots
        U_{n(\ell)}^{\eta_{i(\ell)}}
        \left(
        \begin{smallmatrix}
            b^{(k,\ell,\ell;i)}_{1}\\
            b^{(k,\ell,\ell;i)'}_{1}
        \end{smallmatrix}
        ;\cdots;
        \begin{smallmatrix}
            b^{(k,\ell,\ell;i)}_{n(\ell)}\\
            b^{(k,\ell,\ell;i)'}_{n(\ell)}
        \end{smallmatrix}
        \right)
        (S_{i(\ell)}),
        P(S_{1},\dots,S_{d})
        \Biggr\rangle_{\tau}\\
        &
        =
        \sum_{k=1}^{N}k
        \sum_{\substack{1\leq \ell\leq k\\i\in\mathrm{Alt}(I(\ell,d))}}
        \sum_{\substack{n\in I(\ell,\mathbb{N})\\n(1)+\cdots+n(\ell)=k}}\\
        &\quad\Biggl\langle
        U_{n(1)}^{\eta_{i(1)}}
        \left(
        \begin{smallmatrix}
            b^{(k,\ell,1;i)}_{1}\\
            b^{(k,\ell,1;i)'}_{1}
        \end{smallmatrix}
        ;\cdots;
        \begin{smallmatrix}
            b^{(k,\ell,1;i)}_{n(1)}\\
            b^{(k,\ell,1;i)'}_{n(1)}
        \end{smallmatrix}
        \right)
        (S_{i(1)})\\
        &\hspace{4cm}
        \cdots
        U_{n(\ell)}^{\eta_{i(\ell)}}
        \left(
        \begin{smallmatrix}
            b^{(k,\ell,\ell;i)}_{1}\\
            b^{(k,\ell,\ell;i)'}_{1}
        \end{smallmatrix}
        ;\cdots;
        \begin{smallmatrix}
            b^{(k,\ell,\ell;i)}_{n(\ell)}\\
            b^{(k,\ell,\ell;i)'}_{n(\ell)}
        \end{smallmatrix}
        \right)
        (S_{i(\ell)}),
        P(S_{1},\dots,S_{d})
        \Biggr\rangle_{\tau}.
    \end{align*}
   Since $(S_{1},\dots,S_{d})$ is a $B$-free $B$-valued semi-circular system associated with $\eta_{1},\dots,\eta_{d}$, Propositions \ref{Prop_BChebyshev_tensor} and \ref{Prop_Bsemicircular_equivalent} enable us to compute
    \begin{align*}
        \,&
        \sum_{j=1}^{d}
        \left\|
        \mathrm{ev}_{S}^{\otimes2}
        \left(
        \partial_{j}\left[P(X_{1},\dots,X_{d})\right]
        \right)
        \right\|_{\eta_{j}}^2
        -
        \|P(S_{1},\dots,S_{d})\|^2_{\tau}\\
        &
        =
        \sum_{k=2}^{N}(k-1)
        \sum_{\substack{1\leq \ell\leq k\\i\in\mathrm{Alt}(I(\ell,d))}}
        \sum_{\substack{n\in I(\ell,\mathbb{N})\\n(1)+\cdots+n(\ell)=k}}\\
        &\quad\Biggl\|
        U_{n(1)}^{\eta_{i(1)}}
        \left(
        \begin{smallmatrix}
            b^{(k,\ell,1;i)}_{1}\\
            b^{(k,\ell,1;i)'}_{1}
        \end{smallmatrix}
        ;\cdots;
        \begin{smallmatrix}
            b^{(k,\ell,1;i)}_{n(1)}\\
            b^{(k,\ell,1;i)'}_{n(1)}
        \end{smallmatrix}
        \right)
        (S_{i(1)})
        \cdots
        U_{n(\ell)}^{\eta_{i(\ell)}}
        \left(
        \begin{smallmatrix}
            b^{(k,\ell,\ell;i)}_{1}\\
            b^{(k,\ell,\ell;i)'}_{1}
        \end{smallmatrix}
        ;\cdots;
        \begin{smallmatrix}
            b^{(k,\ell,\ell;i)}_{n(\ell)}\\
            b^{(k,\ell,\ell;i)'}_{n(\ell)}
        \end{smallmatrix}
        \right)
        (S_{i(\ell)})
        \Biggr\|^2_{\tau}\\
        &\geq0,
    \end{align*}
    where the feature of the expression of $P(X_{1},\dots,X_{d})$ (that is, each $(k,\ell,n,i)$-term can appear at most one time) was used in the equality (see Remark \ref{Rem_expression}).
    Thus, we obtain
    \[
    \|P(S_{1},\dots,S_{d})\|^2_{\tau}
    \leq
    \sum_{j=1}^{d}
    \left\|
    \mathrm{ev}_{S}^{\otimes2}
    \left(
    \partial_{j}\left[P(X_{1},\dots,X_{d})\right]
    \right)
    \right\|_{\eta_{j}}^2
    \]
    for any $P\in\mathcal{T}_{\eta}B_{\langle d\rangle}$ with $E[P(S_{1},\dots,S_{d})]=0$.

    (2)$\Rightarrow$(1): By Proposition \ref{Prop_Bsemicircular_equivalent}, it suffices to see that
    \begin{align}
    E
    \left[
    U_{n(1)}^{\eta_{i(1)}}
    \left(
    \begin{smallmatrix}
        b^{(1)}_{1}\\
        b^{(1)'}_{1}
    \end{smallmatrix}
    ;\cdots;
    \begin{smallmatrix}
        b^{(1)}_{n(1)}\\
        b^{(1)'}_{n(1)}
    \end{smallmatrix}
    \right)
    (S_{i(1)})
    \cdots
    U_{n(k)}^{\eta_{i(k)}}
    \left(
    \begin{smallmatrix}
        b^{(k)}_{1}\\
        b^{(k)'}_{1}
    \end{smallmatrix}
    ;\cdots;
    \begin{smallmatrix}
        b^{(k)}_{n(k)}\\
        b^{(k)'}_{n(k)}
    \end{smallmatrix}
    \right)
    (S_{i(k)})
    \right]
    =0 \tag{$*$}
    \end{align}
    for any $k\in\mathbb{N}$, $n\in I(k,\mathbb{N})$, any $i\in \mathrm{Alt}(I(k,d))$ and any $(b^{(j)}_{\ell},b^{(j)'}_{\ell})\in B\times B$ ($j\in[k]$ and $\ell\in[n(j)]$). The case when $n(1)+\cdots +n(k)\leq2$ immediately follows from the assumption that the mean of $S_{j}$ is $0$ and $S_{1},\dots,S_{d}$ are $B$-freely independent. 
    
    Suppose that we have already shown the desired ($*$) when $n(1)+\cdots+n(k)\leq N$ for some $N\in\mathbb{N}$ with $N\geq2$.
    Then, we have to see that
    \[
    E
    \left[
    U_{n(1)}^{\eta_{i(1)}}
    \left(
    \begin{smallmatrix}
        b^{(1)}_{1}\\
        b^{(1)'}_{1}
    \end{smallmatrix}
    ;\cdots;
    \begin{smallmatrix}
        b^{(1)}_{n(1)}\\
        b^{(1)'}_{n(1)}
    \end{smallmatrix}
    \right)
    (S_{i(1)})
    \cdots
    U_{n(k)}^{\eta_{i(k)}}
    \left(
    \begin{smallmatrix}
        b^{(k)}_{1}\\
        b^{(k)'}_{1}
    \end{smallmatrix}
    ;\cdots;
    \begin{smallmatrix}
        b^{(k)}_{n(k)}\\
        b^{(k)'}_{n(k)}
    \end{smallmatrix}
    \right)
    (S_{i(k)})
    \right]
    =0
    \]
    for any $n\in I(k,\mathbb{N})$ with $n(1)+\cdots+n(k)=N+1$. The proof is divided into the following three cases: (\textbf{I}) $n(1)\geq3$, (\textbf{II}) $n(1)=1$, (\textbf{III}) $n(1)=2$.

    (\textbf{I}) If $n(1)\geq3$, then, using the recursion formula in Definition \ref{Def_B_chebyshev}, we have
    \begin{align*}
        \,&
        U_{n(1)}^{\eta_{i(1)}}
        \left(
        \begin{smallmatrix}
            b^{(1)}_{1}\\
            b^{(1)'}_{1}
        \end{smallmatrix}
        ;\cdots;
        \begin{smallmatrix}
            b^{(1)}_{n(1)}\\
            b^{(1)'}_{n(1)}
        \end{smallmatrix}
        \right)
        (S_{i(1)})
        \cdots
        U_{n(k)}^{\eta_{i(k)}}
        \left(
        \begin{smallmatrix}
            b^{(k)}_{1}\\
            b^{(k)'}_{1}
        \end{smallmatrix}
        ;\cdots;
        \begin{smallmatrix}
            b^{(k)}_{n(k)}\\
            b^{(k)'}_{n(k)}
        \end{smallmatrix}
        \right)
        (S_{i(k)})\\
        &=
        U_{1}^{\eta_{i(1)}}
        \left(
        \begin{smallmatrix}
            b^{(1)}_{1}\\
            b^{(1)'}_{1}
        \end{smallmatrix}
        \right)
        (S_{i(1)})
        U_{n(1)-1}^{\eta_{i(1)}}
        \left(
        \begin{smallmatrix}
            b^{(1)}_{2}\\
            b^{(1)'}_{2}
        \end{smallmatrix}
        ;\cdots;
        \begin{smallmatrix}
            b^{(1)}_{n(1)}\\
            b^{(1)'}_{n(1)}
        \end{smallmatrix}
        \right)
        (S_{i(1)})
        \cdots
        U_{n(k)}^{\eta_{i(k)}}
        \left(
        \begin{smallmatrix}
            b^{(k)}_{1}\\
            b^{(k)'}_{1}
        \end{smallmatrix}
        ;\cdots;
        \begin{smallmatrix}
            b^{(k)}_{n(k)}\\
            b^{(k)'}_{n(k)}
        \end{smallmatrix}
        \right)
        (S_{i(k)})\\
        &\quad-b^{(1)}_{1}\eta_{i(1)}\left(b^{(1)'}_{1}b^{(1)}_{2}\right)b^{(1)'}_{2}\\
        &\quad\times
        U_{n(1)-2}^{\eta_{i(1)}}
        \left(
        \begin{smallmatrix}
            b^{(1)}_{3}\\
            b^{(1)'}_{3}
        \end{smallmatrix}
        ;\cdots;
        \begin{smallmatrix}
            b^{(1)}_{n(1)}\\
            b^{(1)'}_{n(1)}
        \end{smallmatrix}
        \right)
        (S_{i(1)})
        \cdots
        U_{n(k)}^{\eta_{i(k)}}
        \left(
        \begin{smallmatrix}
            b^{(k)}_{1}\\
            b^{(k)'}_{1}
        \end{smallmatrix}
        ;\cdots;
        \begin{smallmatrix}
            b^{(k)}_{n(k)}\\
            b^{(k)'}_{n(k)}
        \end{smallmatrix}
        \right)
        (S_{i(k)}).
    \end{align*}
    Since $(n(1)-2)+n(2)+\cdots+n(k)=N-1\leq N$, we have
    \begin{align*}
        \,&
        E
        \left[
        U_{n(1)-2}^{\eta_{i(1)}}
        \left(
        \begin{smallmatrix}
            b^{(1)}_{3}\\
            b^{(1)'}_{3}
        \end{smallmatrix}
        ;\cdots;
        \begin{smallmatrix}
            b^{(1)}_{n(1)}\\
            b^{(1)'}_{n(1)}
        \end{smallmatrix}
        \right)
        (S_{i(1)})
        \cdots
        U_{n(k)}^{\eta_{i(k)}}
        \left(
        \begin{smallmatrix}
            b^{(k)}_{1}\\
            b^{(k)'}_{1}
        \end{smallmatrix}
        ;\cdots;
        \begin{smallmatrix}
            b^{(k)}_{n(k)}\\
            b^{(k)'}_{n(k)}
        \end{smallmatrix}
        \right)
        (S_{i(k)})
        \right]=0.
    \end{align*}
    Therefore, we have to see that
    \begin{align*}
        \,&
        E
        \left[
        U_{1}^{\eta_{i(1)}}
        \left(
        \begin{smallmatrix}
            b^{(1)}_{1}\\
            b^{(1)'}_{1}
        \end{smallmatrix}
        \right)
        (S_{i(1)})
        U_{n(1)-1}^{\eta_{i(1)}}
        \left(
        \begin{smallmatrix}
            b^{(1)}_{2}\\
            b^{(1)'}_{2}
        \end{smallmatrix}
        ;\cdots;
        \begin{smallmatrix}
            b^{(1)}_{n(1)}\\
            b^{(1)'}_{n(1)}
        \end{smallmatrix}
        \right)
        (S_{i(1)})
        \cdots
        U_{n(k)}^{\eta_{i(k)}}
        \left(
        \begin{smallmatrix}
            b^{(k)}_{1}\\
            b^{(k)'}_{1}
        \end{smallmatrix}
        ;\cdots;
        \begin{smallmatrix}
            b^{(k)}_{n(k)}\\
            b^{(k)'}_{n(k)}
        \end{smallmatrix}
        \right)
        (S_{i(k)})
        \right]\\
        &=0.
    \end{align*} 
    
    Let us set
    \begin{align*}
        \,&
        W(X_{1},\dots,X_{d}):=
        U_{n(1)-1}^{\eta_{i(1)}}
        \left(
        \begin{smallmatrix}
            b^{(1)}_{2}\\
            b^{(1)'}_{2}
        \end{smallmatrix}
        ;\cdots;
        \begin{smallmatrix}
            b^{(1)}_{n(1)}\\
            b^{(1)'}_{n(1)}
        \end{smallmatrix}
        \right)
        (X_{i(1)})
        \cdots
        U_{n(k)}^{\eta_{i(k)}}
        \left(
        \begin{smallmatrix}
            b^{(k)}_{1}\\
            b^{(k)'}_{1}
        \end{smallmatrix}
        ;\cdots;
        \begin{smallmatrix}
            b^{(k)}_{n(k)}\\
            b^{(k)'}_{n(k)}
        \end{smallmatrix}
        \right)
        (X_{i(k)})
    \end{align*}
    and
    \begin{align*}
        \,&\widetilde{b}:=
        E
        \left[
        U_{1}^{\eta_{i(1)}}
        \left(
        \begin{smallmatrix}
            b^{(1)}_{1}\\
            b^{(1)'}_{1}
        \end{smallmatrix}
        \right)
        (S_{i(1)})
        W(S_{1},\dots,S_{d})
        \right].
    \end{align*}
    
    Consider
    \[
    \left\|
    a\cdot
    U_{1}^{\eta_{i(1)}}
    \left(
    \begin{smallmatrix}
        \widetilde{b}^*b^{(1)}_{1}\\
        b^{(1)'}_{1}
    \end{smallmatrix}
    \right)
    (S_{i(1)})
    +
    W(S_{1},\dots,S_{d})
    \right\|_{\tau},
    \]
    where $a$ is an arbitrary complex number. Then, by assumption (2), we have
    \begin{align*}
        \,&
        \left\|
        a\cdot
        U_{1}^{\eta_{i(1)}}
        \left(
        \begin{smallmatrix}
            \widetilde{b}^*b^{(1)}_{1}\\
            b^{(1)'}_{1}
        \end{smallmatrix}
        \right)
        (S_{i(1)})
        +
        W(S_{1},\dots,S_{d})
        \right\|_{\tau}^2\\
        &\leq
        \sum_{j=1}^{d}
        \left\|
        \mathrm{ev}_{S}^{\otimes2}
        \left(
        \partial_{j}
        \left[
        a\cdot
        U_{1}^{\eta_{i(1)}}
        \left(
        \begin{smallmatrix}
            \widetilde{b}^*b^{(1)}_{1}\\
            b^{(1)'}_{1}
        \end{smallmatrix}
        \right)
        (X_{i(1)})
        +
        W(X_{1},\dots,X_{d})
        \right]
        \right)
        \right\|_{\eta_{j}}^2.
    \end{align*}
    The left-hand side is equal to
    \begin{align*}
        \,&
        |a|^2\cdot
        \left\|
        U_{1}^{\eta_{i(1)}}
        \left(
        \begin{smallmatrix}
            \widetilde{b}^*b^{(1)}_{1}\\
            b^{(1)}_{1}
        \end{smallmatrix}
        \right)
        (S_{i(1)})
        \right\|^2_{\tau}+
        \left\|
        W(S_{1},\dots,S_{d})
        \right\|_{\tau}^2\\
        &\quad+
        2\mathrm{Re}
        \left(
        a\cdot
        \left\langle
        W(S_{1},\dots,S_{d})^*,
        U_{1}^{\eta_{i(1)}}
        \left(
        \begin{smallmatrix}
            \widetilde{b}^*b^{(1)}_{1}\\
            b^{(1)'}_{1}
        \end{smallmatrix}
        \right)
        (S_{i(1)})
        \right\rangle_{\tau}
        \right),
    \end{align*}
    and the right-hand side is equal to
    \begin{align*}
        \,&
        \sum_{j=1}^{d}
        \Biggl(
        |a|^2\cdot
        \left\|
        \mathrm{ev}_{S}^{\otimes2}
        \left(
        \partial_{j}
        \left[
        U_{1}^{\eta_{i(1)}}
        \left(
        \begin{smallmatrix}
            \widetilde{b}^*b^{(1)}_{1}\\
            b^{(1)'}_{1}
        \end{smallmatrix}
        \right)
        (X_{i(1)})
        \right]
        \right)
        \right\|_{\eta_{j}}^2\\
        &+
        \left\|
        \mathrm{ev}_{S}^{\otimes2}
        \left(
        \partial_{j}
        \left[
        W(X_{1},\dots,X_{d})
        \right]
        \right)
        \right\|_{\eta_{j}}^2\\
        &+2\mathrm{Re}
        \left(
        a\cdot
        \left\langle
        \mathrm{ev}_{S}^{\otimes2}
        \bigl(
        \partial_{j}
        \left[
        W(X_{1},\dots,X_{d})
        \right]
        \bigr)^*,
        \mathrm{ev}_{S}^{\otimes2}
        \left(
        \partial_{j}
        \left[
        U_{1}^{\eta_{i(1)}}
        \left(
        \begin{smallmatrix}
            \widetilde{b}^*b^{(1)}_{1}\\
            b^{(1)}_{1}
        \end{smallmatrix}
        \right)
        (X_{i(1)})
        \right]
        \right)
        \right\rangle_{\eta_{j}}
        \right)
        \Biggr)\\
        &=
        |a|^2\cdot
        \left\|
        \mathrm{ev}_{S}^{\otimes2}
        \left(
        \partial_{i(1)}
        \left[
        U_{1}^{\eta_{i(1)}}
        \left(
        \begin{smallmatrix}
            \widetilde{b}^*b^{(1)}_{1}\\
            b^{(1)'}_{1}
        \end{smallmatrix}
        \right)
        (X_{i(1)})
        \right]
        \right)
        \right\|_{\eta_{i(1)}}^2\\
        &\quad+\sum_{j=1}^{d}
        \left\|
        \mathrm{ev}_{S}^{\otimes2}
        \left(
        \partial_{j}
        \left[
        W(X_{1},\dots,X_{d})
        \right]
        \right)
        \right\|_{\eta_{j}}^2\\
        &\quad
        +2\mathrm{Re}
        \left(
        a\cdot
        \left\langle
        \mathrm{ev}_{S}^{\otimes2}
        \bigl(
        \partial_{i(1)}
        \left[
        W(X_{1},\dots,X_{d})
        \right]
        \bigr)^*,
        \mathrm{ev}_{S}^{\otimes2}
        \left(
        \partial_{i(1)}
        \left[
        U_{1}^{\eta_{i(1)}}
        \left(
        \begin{smallmatrix}
            \widetilde{b}^*b^{(1)}_{1}\\
            b^{(1)}_{1}
        \end{smallmatrix}
        \right)
        (X_{i(1)})
        \right]
        \right)
        \right\rangle_{\eta_{i(1)}}
        \right).
    \end{align*}
    Note that
    \begin{align*}
        \left\|
        U_{1}^{\eta_{i(1)}}
        \left(
        \begin{smallmatrix}
            \widetilde{b}^*b^{(1)}_{1}\\
            b^{(1)'}_{1}
        \end{smallmatrix}
        \right)
        (S_{i(1)})
        \right\|^2_{\tau}
        &=
        \left\langle
        \widetilde{b}^*b^{(1)}_{1}S_{i(1)}b^{(1)'}_{1},
        \widetilde{b}^*b^{(1)}_{1}S_{i(1)}b^{(1)'}_{1}
        \right\rangle_{\tau}\\
        &=
        \tau
        \left(
        (b^{(1)'}_{1})^{*}S_{i(1)}(b^{(1)}_{1})^*\widetilde{b}
        \widetilde{b}^*b^{(1)}_{1}S_{i(1)}b^{(1)'}_{1}
        \right)\\
        &=
        \tau
        \left(
        E\left[
        (b^{(1)'}_{1})^{*}S_{i(1)}(b^{(1)}_{1})^*\widetilde{b}
        \widetilde{b}^*b^{(1)}_{1}S_{i(1)}b^{(1)'}_{1}
        \right]\right)\\
        &=\tau\left(
        (b_{1}')^*\eta_{i(1)}\left((b^{(1)}_{1})^*\widetilde{b}\widetilde{b}^*b^{(1)}_{1}\right)b^{(1)'}_{1}
        \right)\\
        &=
        \left\langle
        \widetilde{b}^*b^{(1)}_{1}\otimes b^{(1)'}_{1},
        \widetilde{b}^*b^{(1)}_{1}\otimes b^{(1)}_{1}
        \right\rangle_{\eta_{j}}\\
        &=
        \left\|
        \mathrm{ev}_{S}^{\otimes2}
        \left(
        \partial_{i(1)}
        \left[
        U_{1}^{\eta_{i(1)}}
        \left(
        \begin{smallmatrix}
            \widetilde{b}^*b^{(1)}_{1}\\
            b^{(1)}_{1}
        \end{smallmatrix}
        \right)
        (X_{i(1)})
        \right]
        \right)
        \right\|^2_{\eta_{i(1)}}.
    \end{align*}
    Thus, we have
    \begin{align*}
        \,&
        \left\|
        W(S_{1},\dots,S_{d})
        \right\|_{\tau}^2
        +
        2\mathrm{Re}
        \left(
        a\cdot
        \left\langle
        W(S_{1},\dots,S_{d})^*,
        U_{1}^{\eta_{i(1)}}
        \left(
        \begin{smallmatrix}
            \widetilde{b}^*b^{(1)}_{1}\\
            b^{(1)'}_{1}
        \end{smallmatrix}
        \right)
        (S_{i(1)})
        \right\rangle_{\tau}
        \right)\\
        &\leq
        \sum_{j=1}^{d}
        \left\|
        \mathrm{ev}_{S}^{\otimes2}
        \left(
        \partial_{j}
        \left[
        W(X_{1},\dots,X_{d})
        \right]
        \right)
        \right\|_{\eta_{j}}^2\\
        &\quad
        +
        2\mathrm{Re}
        \left(
        a\cdot
        \left\langle
        \mathrm{ev}_{S}^{\otimes2}
        \bigl(
        \partial_{i(1)}
        \left[
        W(X_{1},\dots,X_{d})
        \right]
        \bigr)^*,
        \mathrm{ev}_{S}^{\otimes2}
        \left(
        \partial_{i(1)}
        \left[
        U_{1}^{\eta_{i(1)}}
        \left(
        \begin{smallmatrix}
            \widetilde{b}^*b^{(1)}_{1}\\
            b^{(1)}_{1}
        \end{smallmatrix}
        \right)
        (X_{i(1)})
        \right]
        \right)
        \right\rangle_{\eta_{i(1)}}
        \right).
    \end{align*}
    
    By assumption (2), we also have
    \begin{align*}
        \,&
        \left\|
        W(S_{1},\dots,S_{d})
        \right\|_{\tau}^2
        \leq
        \sum_{j=1}^{d}
        \left\|
        \mathrm{ev}_{S}^{\otimes2}
        \left(
        \partial_{j}
        \left[
        W(X_{1},\dots,X_{d})
        \right]
        \right)
        \right\|_{\eta_{j}}^2,
    \end{align*}
    and hence, considering $a\in\mathrm{R}$ of sufficiently large absolute value, the following holds:
    \begin{align*}
        \,&
        \mathrm{Re}
        \left(
        \left\langle
        W(S_{1},\dots,S_{d})^*,
        U_{1}^{\eta_{i(1)}}
        \left(
        \begin{smallmatrix}
            \widetilde{b}^*b^{(1)}_{1}\\
            b^{(1)'}_{1}
        \end{smallmatrix}
        \right)
        (S_{i(1)})
        \right\rangle_{\tau}
        \right)\\
        &=
        \mathrm{Re}
        \left(
        \left\langle
        \mathrm{ev}_{S}^{\otimes2}
        \bigl(
        \partial_{i(1)}
        \left[
        W(X_{1},\dots,X_{d})
        \right]
        \bigr)^*,
        \mathrm{ev}_{S}^{\otimes2}
        \left(
        \partial_{i(1)}
        \left[
        U_{1}^{\eta_{i(1)}}
        \left(
        \begin{smallmatrix}
            \widetilde{b}^*b^{(1)}_{1}\\
            b^{(1)}_{1}
        \end{smallmatrix}
        \right)
        (X_{i(1)})
        \right]
        \right)
        \right\rangle_{\eta_{i(1)}}
        \right).
    \end{align*}
    Similarly, considering $a\in i\mathbb{R}$ of sufficiently large absolute value, we also have
    \begin{align*}
        \,&
        \mathrm{Im}
        \left(
        \left\langle
        W(S_{1},\dots,S_{d})^*,
        U_{1}^{\eta_{i(1)}}
        \left(
        \begin{smallmatrix}
            \widetilde{b}^*b^{(1)}_{1}\\
            b^{(1)'}_{1}
        \end{smallmatrix}
        \right)
        (S_{i(1)})
        \right\rangle_{\tau}
        \right)\\
        &=
        \mathrm{Im}
        \left(
        \left\langle
        \mathrm{ev}_{S}^{\otimes2}
        \bigl(
        \partial_{i(1)}
        \left[
        W(X_{1},\dots,X_{d})
        \right]
        \bigr)^*,
        \mathrm{ev}_{S}^{\otimes2}
        \left(
        \partial_{i(1)}
        \left[
        U_{1}^{\eta_{i(1)}}
        \left(
        \begin{smallmatrix}
            \widetilde{b}^*b^{(1)}_{1}\\
            b^{(1)}_{1}
        \end{smallmatrix}
        \right)
        (X_{i(1)})
        \right]
        \right)
        \right\rangle_{\eta_{i(1)}}
        \right).
    \end{align*}
    and hence
    \begin{align}
        \,&
        \left\langle
        W(S_{1},\dots,S_{d})^*,
        U_{1}^{\eta_{i(1)}}
        \left(
        \begin{smallmatrix}
            \widetilde{b}^*b^{(1)}_{1}\\
            b^{(1)'}_{1}
        \end{smallmatrix}
        \right)
        (S_{i(1)})
        \right\rangle_{\tau}\\
        &=
        \left\langle
        \mathrm{ev}_{S}^{\otimes2}
        \bigl(
        \partial_{i(1)}
        \left[
        W(X_{1},\dots,X_{d})
        \right]
        \bigr)^*,
        \mathrm{ev}_{S}^{\otimes2}
        \left(
        \partial_{i(1)}
        \left[
        U_{1}^{\eta_{i(1)}}
        \left(
        \begin{smallmatrix}
            \widetilde{b}^*b^{(1)}_{1}\\
            b^{(1)}_{1}
        \end{smallmatrix}
        \right)
        (X_{i(1)})
        \right]
        \right)
        \right\rangle_{\eta_{i(1)}} \tag{$**$}.
    \end{align}
    Therefore, we obtain
    \begin{align*}
        \,&
        \tau\left(
        \widetilde{b}^*\widetilde{b}
        \right)\\
        &=
        \tau
        \left(
        E\left[
        U_{1}^{\eta_{i(1)}}
        \left(
        \begin{smallmatrix}
            \widetilde{b}^*b^{(1)}_{1}\\
            b^{(1)}_{1}
        \end{smallmatrix}
        \right)
        (S_{i(1)})
        W(S_{1},\dots,S_{d})
        \right]
        \right)\\
        &=
        \tau\left(
        U_{1}^{\eta_{i(1)}}
        \left(
        \begin{smallmatrix}
            \widetilde{b}^*b^{(1)}_{1}\\
            b^{(1)}_{1}
        \end{smallmatrix}
        \right)
        (S_{i(1)})
        W(S_{1},\dots,S_{d})
        \right)\\
        &=
        \tau\left(
        W(S_{1},\dots,S_{d})
        U_{1}^{\eta_{i(1)}}
        \left(
        \begin{smallmatrix}
            \widetilde{b}^*b^{(1)}_{1}\\
            b^{(1)}_{1}
        \end{smallmatrix}
        \right)
        (S_{i(1)})
        \right)\\
        &=\left\langle
        W(S_{1},\dots,S_{d})^*,
        U_{1}^{\eta_{i(1)}}
        \left(
        \begin{smallmatrix}
            \widetilde{b}^*b^{(1)}_{1}\\
            b^{(1)}_{1}
        \end{smallmatrix}
        \right)
        (S_{i(1)})
        \right\rangle_{\tau}\\
        &=
        \left\langle
        \mathrm{ev}_{S}^{\otimes2}
        \bigl(
        \partial_{i(1)}
        \left[
        W(X_{1},\dots,X_{d})
        \right]
        \bigr)^*,
        \mathrm{ev}_{S}^{\otimes2}
        \left(
        \partial_{i(1)}
        \left[
        U_{1}^{\eta_{i(1)}}
        \left(
        \begin{smallmatrix}
            \widetilde{b}^*b^{(1)}_{1}\\
            b^{(1)}_{1}
        \end{smallmatrix}
        \right)
        (X_{i(1)})
        \right]
        \right)
        \right\rangle_{\eta_{i(1)}}\\
        &=
        \left\langle
        \mathrm{ev}_{S}^{\otimes2}
        \bigl(
        \partial_{i(1)}
        \left[
        W(X_{1},\dots,X_{d})
        \right]
        \bigr)^*,
        \widetilde{b}^*b^{(1)}_{1}\otimes b^{(1)'}_{1}
        \right\rangle_{\eta_{i(1)}},
    \end{align*}
    where the $\tau$-preserving property of $E$, the tracial property of $\tau$ and the formula ($**$) were used in the second, the third and the fifth equalities, respectively. 
    Remark that
    \begin{align*}
        \,&
        \partial_{i(1)}
        \left[
        W(X_{1},\dots,X_{d})
        \right]\\
        &=
        \partial_{i(1)}
        \left[
        U_{n(1)-1}^{\eta_{i(1)}}
        \left(
        \begin{smallmatrix}
            b^{(1)}_{2}\\
            b^{(1)'}_{2}
        \end{smallmatrix}
        ;\cdots;
        \begin{smallmatrix}
            b^{(1)}_{n(1)}\\
            b^{(1)'}_{n(1)}
        \end{smallmatrix}
        \right)
        (X_{i(1)})
        \right]\\
        &\quad\times
        \left(
        1\otimes
        U_{n(2)}^{\eta_{i(2)}}
        \left(
        \begin{smallmatrix}
            b^{(2)}_{1}\\
            b^{(2)'}_{1}
        \end{smallmatrix}
        ;\cdots;
        \begin{smallmatrix}
            b^{(2)}_{n(2)}\\
            b^{(2)'}_{n(2)}
        \end{smallmatrix}
        \right)
        (X_{i(2)})
        \cdots
        U_{n(k)}^{\eta_{i(k)}}
        \left(
        \begin{smallmatrix}
            b^{(k)}_{1}\\
            b^{(k)'}_{1}
        \end{smallmatrix}
        ;\cdots;
        \begin{smallmatrix}
            b^{(k)}_{n(k)}\\
            b^{(k)'}_{n(k)}
        \end{smallmatrix}
        \right)
        (X_{i(k)})
        \right)\\
        &+\sum_{\substack{2\leq m\leq k\\i(m)=i(1)}}
        \Biggl(
        U_{n(1)-1}^{\eta_{i(1)}}
        \left(
        \begin{smallmatrix}
            b^{(1)}_{2}\\
            b^{(1)'}_{2}
        \end{smallmatrix}
        ;\cdots;
        \begin{smallmatrix}
            b^{(1)}_{n(1)}\\
            b^{(1)'}_{n(1)}
        \end{smallmatrix}
        \right)
        (X_{i(1)})
        U_{n(2)}^{\eta_{i(2)}}
        \left(
        \begin{smallmatrix}
            b^{(2)}_{1}\\
            b^{(2)'}_{1}
        \end{smallmatrix}
        ;\cdots;
        \begin{smallmatrix}
            b^{(2)}_{n(2)}\\
            b^{(2)'}_{n(2)}
        \end{smallmatrix}
        \right)
        (X_{i(2)})
        \cdots\\
        &\hspace{5cm}\cdots
        U_{n(m-1)}^{\eta_{i(m-1)}}
        \left(
        \begin{smallmatrix}
            b^{(m-1)}_{1}\\
            b^{(m-1)'}_{1}
        \end{smallmatrix}
        ;\cdots;
        \begin{smallmatrix}
            b^{(m-1)}_{n(m-1)}\\
            b^{(m-1)'}_{n(m-1)}
        \end{smallmatrix}
        \right)
        (X_{i(m-1)})
        \otimes1
        \Biggr)\\
        &\quad\times
        \partial_{i(1)}
        \left[
        U_{n(m)}^{\eta_{i(m)}}
        \left(
        \begin{smallmatrix}
            b^{(m)}_{1}\\
            b^{(m)'}_{1}
        \end{smallmatrix}
        ;\cdots;
        \begin{smallmatrix}
            b^{(m)}_{n(m)}\\
            b^{(m)'}_{n(m)}
        \end{smallmatrix}
        \right)
        (X_{i(m)})
        \right]\\
        &\quad\times
        \left(
        1\otimes
        U_{n(m+1)}^{\eta_{i(m+1)}}
        \left(
        \begin{smallmatrix}
            b^{(m+1)}_{1}\\
            b^{(m+1)'}_{1}
        \end{smallmatrix}
        ;\cdots;
        \begin{smallmatrix}
            b^{(m+1)}_{n(m+1)}\\
            b^{(m+1)'}_{n(m+1)}
        \end{smallmatrix}
        \right)
        (X_{i(m+1)})
        \cdots
        U_{n(k)}^{\eta_{i(k)}}
        \left(
        \begin{smallmatrix}
            b^{(k)}_{1}\\
            b^{(k)'}_{1}
        \end{smallmatrix}
        ;\cdots;
        \begin{smallmatrix}
            b^{(k)}_{n(k)}\\
            b^{(k)'}_{n(k)}
        \end{smallmatrix}
        \right)
        (X_{i(k)})
        \right).
    \end{align*}
    By Proposition \ref{Prop_fromula_fdq_Chebyshev} and the induction hypothesis, it easily follows that
    \[
    \tau\left(\widetilde{b}^*\widetilde{b}\right)
    =
    \left\langle
    \mathrm{ev}_{S}^{\otimes2}
    \bigl(
    \partial_{i(1)}
    \left[
    W(X_{1},\dots,X_{d})
    \right]
    \bigr)^*,
    \widetilde{b}^*b^{(1)}_{1}\otimes b^{(1)'}_{1}
    \right\rangle_{\eta_{i(1)}}
    =0.
    \]
    By the faithfulness, we obtain
    \begin{align*}
        \,&
        E
        \left[
        U_{1}^{\eta_{i(1)}}
        \left(
        \begin{smallmatrix}
            b^{(1)}_{1}\\
            b^{(1)'}_{1}
        \end{smallmatrix}
        \right)
        (S_{i(1)})
        W(S_{1},\dots,S_{d})
        \right]=0
    \end{align*}
    as desired, where
    \[
    W(X_{1},\dots,X_{d})
    =
    U_{n(1)-1}^{\eta_{i(1)}}
    \left(
    \begin{smallmatrix}
        b^{(1)}_{2}\\
        b^{(1)'}_{2}
    \end{smallmatrix}
    ;\cdots;
    \begin{smallmatrix}
        b^{(1)}_{n(1)}\\
        b^{(1)'}_{n(1)}
    \end{smallmatrix}
    \right)
    (X_{i(1)})
    \cdots
    U_{n(k)}^{\eta_{i(k)}}
    \left(
    \begin{smallmatrix}
        b^{(k)}_{1}\\
        b^{(k)'}_{1}
    \end{smallmatrix}
    ;\cdots;
    \begin{smallmatrix}
        b^{(k)}_{n(k)}\\
        b^{(k)'}_{n(k)}
    \end{smallmatrix}
    \right)
    (X_{i(k)}).
    \]
    Thus, the proof when $n(1)\geq3$ has been completed by induction.

    (\textbf{II}) Assume $n(1)=1$. Let us set
    \[
    Z(X_{1},\dots,X_{d})
    =
    U_{n(2)}^{\eta_{i(2)}}
    \left(
    \begin{smallmatrix}
        b^{(2)}_{1}\\
        b^{(2)'}_{1}
    \end{smallmatrix}
    ;\cdots;
    \begin{smallmatrix}
        b^{(2)}_{n(2)}\\
        b^{(2)'}_{n(2)}
    \end{smallmatrix}
    \right)
    (X_{i(2)})
    \cdots
    U_{n(k)}^{\eta_{i(k)}}
    \left(
    \begin{smallmatrix}
        b^{(k)}_{1}\\
        b^{(k)'}_{1}
    \end{smallmatrix}
    ;\cdots;
    \begin{smallmatrix}
        b^{(k)}_{n(k)}\\
        b^{(k)'}_{n(k)}
    \end{smallmatrix}
    \right)
    (X_{i(k)})
    \]
    and
    \[
    \widetilde{c}
    =
    E\left[
    U_{1}^{\eta_{i(1)}}
    \left(
    \begin{smallmatrix}
        b^{(1)}_{1}\\
        b^{(1)'}_{1}
    \end{smallmatrix}
    \right)
    (S_{i(1)})
    Z(S_{1},\dots,S_{d})
    \right].
    \]
    Then, we have to see that $\widetilde{c}=0$. This follows from the argument of (\textbf{I}) with replacing $\widetilde{b}$ and 
    \[
    \left\|
    a\cdot
    U_{1}^{\eta_{i(1)}}
    \left(
    \begin{smallmatrix}
        \widetilde{b}^*b^{(1)}_{1}\\
        b^{(1)'}_{1}
    \end{smallmatrix}
    \right)
    (S_{i(1)})
    +
    W(S_{1},\dots,S_{d})
    \right\|_{\tau}
    \]
    with $\widetilde{c}$ and 
    \[
    \left\|
    a\cdot
    U_{1}^{\eta_{i(1)}}
    \left(
    \begin{smallmatrix}
        \widetilde{c}^*b^{(1)}_{1}\\
        b^{(1)'}_{1}
    \end{smallmatrix}
    \right)
    (S_{i(1)})
    +
    Z(S_{1},\dots,S_{d})
    \right\|_{\tau},
    \]
    respectively.

    (\textbf{III}) The proof in the case of $n(1)=2$ is similar to that of the case when $n(1)\geq3$, since
    \[
    U_{2}^{\eta_{i(1)}}
    \left(
    \begin{smallmatrix}
        b^{(1)}_{1}\\
        b^{(1)'}_{1}
    \end{smallmatrix}
    ;
    \begin{smallmatrix}
        b^{(1)}_{2}\\
        b^{(1)'}_{2}
    \end{smallmatrix}
    \right)
    (X_{i(1)})
    =
    U_{1}^{\eta_{i(1)}}
    \left(
    \begin{smallmatrix}
        b^{(1)}_{1}\\
        b^{(1)'}_{1}
    \end{smallmatrix}
    \right)
    (X_{i(1)})
    U_{1}^{\eta_{i(1)}}
    \left(
    \begin{smallmatrix}
        b^{(1)}_{2}\\
        b^{(1)'}_{2}
    \end{smallmatrix}
    \right)
    (X_{i(1)})
    -b^{(1)}_{1}\eta\left(b^{(1)'}_{1}b^{(1)}_{2}\right)b^{(1)'}_{1}
    \]
    and
    \begin{align*}
        \,&
        U_{2}^{\eta_{i(1)}}
        \left(
        \begin{smallmatrix}
            b^{(1)}_{1}\\
            b^{(1)'}_{1}
        \end{smallmatrix}
        ;\cdots;
        \begin{smallmatrix}
            b^{(1)}_{n(1)}\\
            b^{(1)'}_{n(1)}
        \end{smallmatrix}
        \right)
        (S_{i(1)})
        U_{n(2)}^{\eta_{i(2)}}
        \left(
        \begin{smallmatrix}
            b^{(2)}_{1}\\
            b^{(2)'}_{1}
        \end{smallmatrix}
        ;\cdots;
        \begin{smallmatrix}
            b^{(2)}_{n(1)}\\
            b^{(2)'}_{n(1)}
        \end{smallmatrix}
        \right)
        (S_{i(2)})\cdots\\
        &\hspace{8cm}
        \cdots
        U_{n(k)}^{\eta_{i(k)}}
        \left(
        \begin{smallmatrix}
            b^{(k)}_{1}\\
            b^{(k)'}_{1}
        \end{smallmatrix}
        ;\cdots;
        \begin{smallmatrix}
            b^{(k)}_{n(k)}\\
            b^{(k)'}_{n(k)}
        \end{smallmatrix}
        \right)
        (S_{i(k)})\\
        &=
        U_{1}^{\eta_{i(1)}}
        \left(
        \begin{smallmatrix}
            b^{(1)}_{1}\\
            b^{(1)'}_{1}
        \end{smallmatrix}
        \right)
        (X_{i(1)})
        U_{1}^{\eta_{i(1)}}
        \left(
        \begin{smallmatrix}
            b^{(1)}_{2}\\
            b^{(1)'}_{2}
        \end{smallmatrix}
        \right)
        (X_{i(1)})
        U_{n(2)}^{\eta_{i(2)}}
        \left(
        \begin{smallmatrix}
            b^{(2)}_{1}\\
            b^{(2)'}_{1}
        \end{smallmatrix}
        ;\cdots;
        \begin{smallmatrix}
            b^{(2)}_{n(1)}\\
            b^{(2)'}_{n(1)}
        \end{smallmatrix}
        \right)
        (S_{i(2)})\cdots\\
        &\hspace{8cm}
        \cdots
        U_{n(k)}^{\eta_{i(k)}}
        \left(
        \begin{smallmatrix}
            b^{(k)}_{1}\\
            b^{(k)'}_{1}
        \end{smallmatrix}
        ;\cdots;
        \begin{smallmatrix}
            b^{(k)}_{n(k)}\\
            b^{(k)'}_{n(k)}
        \end{smallmatrix}
        \right)
        (S_{i(k)})\\
        &-b^{(1)}_{1}\eta\left(b^{(1)'}_{1}b^{(1)}_{2}\right)b^{(1)'}_{1}
        U_{n(2)}^{\eta_{i(2)}}
        \left(
        \begin{smallmatrix}
            b^{(2)}_{1}\\
            b^{(2)'}_{1}
        \end{smallmatrix}
        ;\cdots;
        \begin{smallmatrix}
            b^{(2)}_{n(1)}\\
            b^{(2)'}_{n(1)}
        \end{smallmatrix}
        \right)
        (S_{i(2)})
        \cdots
        U_{n(k)}^{\eta_{i(k)}}
        \left(
        \begin{smallmatrix}
            b^{(k)}_{1}\\
            b^{(k)'}_{1}
        \end{smallmatrix}
        ;\cdots;
        \begin{smallmatrix}
            b^{(k)}_{n(k)}\\
            b^{(k)'}_{n(k)}
        \end{smallmatrix}
        \right)
        (S_{i(k)}).
    \end{align*}
    Therefore, the proof is completed in all the cases. 
    
    So far, we have proved (1)$\Leftrightarrow$(2) for any $B$-free $B$-valued semi-circular system with respect to $\eta_{1},\dots,\eta_{d}$ such that $\tau(\eta_{j}(b)b')=\tau(b\eta_{j}(b'))$ for any $b,b'\in B$. 

    (3)$\Rightarrow$(2): Trivial.

    (2)$\Rightarrow$(3): By Proposition \ref{prop_chebyshevorthog_formla}, for any $P\in B_{\langle d\rangle}$ there exist $N\in\mathbb{N}$ and $\widetilde{P}\in \mathcal{T}_{\eta\otimes\mathrm{id}_{N}}(M_{N}(B))_{\langle d\rangle}$ such that 
    $\left[\begin{smallmatrix}
        P(X_{1},\dots,X_{d})&\\
        &\mbox{\Large 0}_{N-1}
    \end{smallmatrix}\right]
    =\mathrm{ev}_{X\otimes I_{N}}(\widetilde{P})$ and $\left\|\mathrm{ev}_{S\otimes I_{N}}^{\otimes2}(\partial_{Y_{j}}\widetilde{P})\right\|_{\eta_{j}\otimes\mathrm{id}_{N}}=\frac{1}{N}\left\|\mathrm{ev}_{S}^{\otimes2}(\partial_{j}P)\right\|^2_{\eta_{j}}$.
    Then, applying the equivalence between (1) and (2) to the $M_{N}(B)$-free $M_{N}(B)$-valued semi-circular system $S\otimes I_{N}:=(S_{1}\otimes I_{N},\dots,S_{d}\otimes I_{N})$ with respect to $\eta\otimes\mathrm{id}_{N}=(\eta_{1}\otimes \mathrm{id}_{N},\dots,\eta_{d}\otimes\mathrm{id}_{N})$ (see Remark \ref{rem_bncpspre}), we observe that
    \begin{align*}
        \,&
        \frac{1}{k}
        \left\|
        \begin{bmatrix}
            P(S_{1},\dots,S_{d})-E[P(S_{1},\dots,S_{d})]
        \end{bmatrix}
        \right\|_{\tau}^2\\
        &
        =\left\|
        \begin{bmatrix}
            P(S_{1},\dots,S_{d})-E[P(S_{1},\dots,S_{d})]&\\
            &\mbox{\Large 0}_{N-1}
        \end{bmatrix}
        \right\|_{\tau\otimes\mathrm{tr}_{N}}^2\\
        &=\left\|\widetilde{P}(S_{1}\otimes I_{N},\dots,S_{d}\otimes I_{N})-(E\otimes \mathrm{id}_{N})[\widetilde{P}(S_{1}\otimes I_{N},\dots,S_{d}\otimes I_{N})]\right\|_{\tau\otimes\mathrm{tr}_{N}}^2\\
        &
        \leq\sum_{j=1}^{d}
        \left\|
        \mathrm{ev}_{S\otimes I_{N}}^{\otimes2}(\partial_{Y_{j}}\widetilde{P}(Y_{1},\dots,Y_{d}))
        \right\|_{\eta_{j}\otimes\mathrm{id}_{N}}^2,
    \end{align*}
    where the symbol $\|\cdot\|_{\eta_{j}\otimes\mathrm{id}_{N}}$ is the norm with respect to the (pre-)inner product $\langle\cdot,\cdot\rangle_{\eta_{j}\otimes\mathrm{id}_{N}}$.
    Using Lemma \ref{lem_matampChebyshev1}, we have 
    \[\left\|
    \mathrm{ev}_{S\otimes I_{N}}^{\otimes2}(\partial_{Y_{j}}[\widetilde{P}(Y_{1},\dots,Y_{d})])
    \right\|_{\eta_{j}\otimes\mathrm{id}_{N}}^2
    =\frac{1}{N}
    \|\mathrm{ev}_{S}^{\otimes2}(\partial_{}[P(X_{1},\dots,X_{d})])\|_{\eta_{j}}^2,
    \]
    and hence
    \[
    \left\|
    P(S_{1},\dots,S_{d})-E[P(S_{1},\dots,S_{d})]
    \right\|_{\tau}^2
    \leq
    \sum_{j=1}^{d}
    \|\mathrm{ev}_{S}^{\otimes 2}(\partial_{j}[P(X_{1},\dots,X_{d})])\|_{\eta_{j}}^2.
    \]
    Thus, we are done.
\end{proof}

\begin{Rem}
    In the case of $B=\mathbb{C}$ and $\eta_{1}(1)=\cdots=\eta_{d}(1)=1$, Theorem \ref{Thm_character_Bsemicircular} is nothing but Biane's theorem \cite[Theorem 5.1]{b03}.
\end{Rem}

\subsection{The kernel of the free difference quotient associated with variance with respect to $B$-valued semi-circular system}
Let $(A,B,\tau,E)$ be a tracial $B$-valued $C^*$-probability space and $S=(S_{1},\dots,S_{d})\in A_{\mathrm{sa}}^d$ a $B$-free $B$-valued semi-circular system with respect to $\eta=(\eta_{1},\dots,\eta_{d})$ (with mean $0$). We set $B_{\langle d|S\rangle}=\mathrm{ev}_{S}(B_{\langle d\rangle})$. 
Define Hilbert spaces
\[
L^2(B_{\langle d|S\rangle},\tau)=\overline{B_{\langle d|S\rangle}}^{\langle\cdot,\cdot\rangle_{\tau}},\quad
L^2(B_{\langle d|S\rangle}^{\otimes2},\tau,\eta_{j})=\overline{\iota_{\eta_{j}}\left(B_{\langle d|S\rangle}\otimes B_{\langle d|S\rangle}\right)}^{\langle\cdot,\cdot\rangle_{\eta_{j}}},
\]
where $\iota_{\eta_{j}}$ denotes the quotient map $B_{\langle d|S\rangle}\otimes B_{\langle d|S\rangle}\to\quotient{B_{\langle d|S}\otimes B_{\langle d|S\rangle}}{\mathcal{N}}$ with $\mathcal{N}=\{\Xi\in B_{\langle d|S\rangle}\otimes B_{\langle d|S\rangle}\,|\,\langle\Xi,\Xi\rangle_{\eta_{j}}=0\}$, and
\[
 L^2(B_{\langle d|S\rangle}^{\otimes2},\tau,\eta)=\bigoplus_{j=1}^{d}L^2(B_{\langle d|S\rangle}^{\otimes2},\tau,\eta_{j})\mbox{ with inner product }
 \left\langle
 \left[
 \begin{smallmatrix}
     \Xi_{1}\\
     \vdots\\
     \Xi_{d}
 \end{smallmatrix}\right],
 \left[
 \begin{smallmatrix}
     \Xi'_{1}\\
     \vdots\\
     \Xi'_{d}
 \end{smallmatrix}
 \right]
 \right\rangle_{\eta}
 =\sum_{j=1}^{d}\langle\Xi_{j},\Xi_{j}'\rangle_{\eta_{j}}
\]
for any $\Xi_{j},\Xi_{j}'\in L^2(B_{\langle d|S\rangle}^{\otimes2},\tau,\eta_{j})$.
Similarly to \cite[Remark 2.3]{l24}, we can see the next lemma.
\begin{Lem}
    In the same notation as above, we have the following:
    \begin{enumerate}
        \item $\ker(\mathrm{ev}_{S})\subset\ker(\iota_{\eta_{j}}\circ\mathrm{ev}_{S}^{\otimes2}\circ\partial_{j})$ for any $j\in[d]$.
        \item For any $j\in[d]$, there is a densely defined unbounded operator
        \[\partial_{S_{j},\eta_{j}}:L^2(B_{\langle d|S\rangle},\tau)\to L^2(B_{\langle d|S\rangle}^{\otimes2},\tau,\eta_{j})\] with domain $B_{\langle d|S\rangle}\simeq\quotient{B_{\langle d\rangle}}{\ker(\mathrm{ev}_{S})}$ such that
        \begin{equation*}
        \xymatrix{
        &B_{\langle d\rangle}\ar[ld]_-{\mathrm{ev}_{S}}\ar[d]_-{\pi}\ar[rd]^-{\iota_{\eta_{j}}\circ\mathrm{ev}_{S}^{\otimes2}\circ\partial_{j}}\ar@{}@<2.0ex>[ld]|{\circlearrowright}&\\
        B_{\langle d|S\rangle}\ar[r]_-{\sim}&\quotient{B_{\langle d\rangle}}{\ker(\mathrm{ev}_{S})}\ar[r]_{\quad\partial_{S_{j},\eta_{j}}}&\iota_{\eta_{j}}(B_{\langle d|S\rangle}^{\otimes2}),\ar@{}@<2.0ex>[lu]|{\circlearrowright}
        }
        \end{equation*}
        where $\pi:B_{\langle d\rangle}\to\quotient{B_{\langle d\rangle}}{\ker(\mathrm{ev}_{S})}$ is the quotient map.
    \end{enumerate}
\end{Lem}

Using Corollary \ref{Cor_integralbyparts}, it is easy to see the following:
\begin{Cor}
    The adjoint $\partial_{S_{j},\eta_{j}}^*$ of $\partial_{S_{j},\eta_{j}}$ is given by $\partial^*_{S_{j},\eta_{j}}(\mathrm{ev}_{S}(\Xi))=\mathrm{ev}_{S}(\partial_{j}^*\Xi)$ for any $\Xi\in B_{\langle d\rangle}^{\otimes2}$, and $\iota_{\eta_{j}}(B_{\langle d|S\rangle}^{\otimes2})\subset\mathrm{dom}(\partial_{S_{j},\eta_{j}}^*)$. Hence, $\partial_{S_{j},\eta_{j}}$ is closable.
\end{Cor}

We denote by $\overline{\partial}_{S_{j},\eta_{j}}$ the closure of $\partial_{S_{j},\eta_{j}}:L^{2}(B_{\langle d|S\rangle},\tau)$ and $L^2(B_{\langle d|S\rangle}^{\otimes2},\tau,\eta_{j})$. We also set a closable operator $\partial_{S,\eta}:L^2(B_{\langle d|S\rangle,\tau})\to L^2(B_{\langle d|S\rangle},\tau,\eta)$, whose closure is denoted by $\overline{\partial}_{S,\eta}$, with domain $B_{\langle d|S\rangle}$ such that
$\partial_{S,\eta}\xi=\left(\partial_{S_{1},\eta_{1}}\xi,\dots,\partial_{S_{d},\eta_{d}}\xi\right)$ for any $\xi\in B_{\langle d|S\rangle}$. 
By Theorem \ref{Thm_character_Bsemicircular}(3), we have the following:

\begin{Cor}\label{cor_kernel}
    We have
    \[
    \|\xi-E[\xi]\|_{\tau}^2\leq\sum_{j=1}^{d}\|\overline{\partial}_{S_{j},\eta_{j}}\xi\|_{\eta_{j}}^2=\|\overline{\partial}_{S,\eta}\xi\|_{\eta}^2
    \]
    for any $\xi\in\mathrm{dom}(\overline{\partial}_{S,\eta})$, and hence $\ker(\overline{\partial}_{S,\eta})=L^2(B,\tau)$.
\end{Cor}

\begin{Rem}
    Let $\tau$ be a faithful tracial state. Consider the case when $d=1$ and $\eta(b)=\eta_{1}(b)=\tau(b)1$ for any $b\in B$, that is, $B$ and $S=S_{1}$ are $\mathbb{C}$-free with respect to $\tau$ (see \cite[Corollary 2.5]{sh98}). In this case, since $\|\cdot\|_{\eta}^2=\|\cdot\|^2_{\tau\otimes\tau}$, we have 
    \[
    \|\xi-E[\xi]\|_{\tau}^2\leq\|\overline{\partial}_{S,\tau}\xi\|^2_{\tau\otimes\tau}
    \]
    for any $\xi\in\mathrm{dom}(\overline{\partial}_{S,\tau})$, and hence $\ker(\overline{\partial}_{S,\tau})=L^2(B,\tau)$. This gives an affirmative answer to Voiculescu's conjecture \cite{aim06} in a very particular case. Probably, this is a folklore, but we could not find any literature asserting even this.
\end{Rem}

\section{A counterexample to Voiculescu's conjecture on $B$-valued free Poincar\'{e} inequality}\label{section_counterexample}
\subsection{Counterexample}
In this section, we will give a counterexample to Voiculescu's conjecture on $B$-valued free Poincar\'{e} inequality (see \cite[section 0.1]{aim06}).

Let $A\supset B$ be a unital inclusion of unital $C^*$-algebras (resp. $W^*$-algebras) and $(A,B,\tau,E)$ be a tracial $B$-valued $C^*$-probability space (resp. $W^*$-probability space). Assume that there exists a family $(e_{m})_{m=1}^{\infty}$ of projections in $B$ such that $e_{n}e_{m}=\delta_{n,m}e_{n}$ and $\tau\left(e_{n}\right)=\frac{6}{\pi^2}\frac{1}{n^2}$ for each $n\in\mathbb{N}$. (This happens if, for example, $B$ is diffuse, or $B=c_{0}(\mathbb{N})^{\sim}$ with a faithful tracial state $\tau\left(x\right)=\frac{6}{\pi^2}\sum_{n=1}^{\infty}\frac{x_{n}}{n^2}$, $x=(x_{n})_{n=1}^{\infty}\in c_{0}(\mathbb{N})^{\sim}$.) Let also $a$ be a self-adjoint $B$-valued random variable with mean $0$ and variance $\eta:B\to B$ such that $\eta(b)=b$, $b\in B$.

The next lemma is probably well known, but we do give it explicitly for the sake of completeness.

\begin{Lem}\label{Lem_commu}
    We have $ab=ba$ for any $b\in B$.
\end{Lem}
\begin{proof}
    Take an arbitrary element $b\in B$. We observe that
    \begin{align*}
        \tau\left((ab-ba)^*(ab-ba)\right)
        &=
        \tau\left(b^*aab-b^*aba-ab^*ab+ab^*ba\right)\\
        &=
        \tau\left(b^*E[a1a]b-b^*E[aba]-E[ab^*a]b^*+E[ab^*ba]\right)\\
        &=
        \tau
        \left(
        b^*\eta(1)b
        -b^*\eta(b)
        -\eta(b^*)b
        +\eta(b^*b)
        \right)\\
        &=
        \tau
        \left(
        b^*b
        -b^*b
        -b^*b
        +b^*b
        \right)=0,
    \end{align*}
    where the $\tau$-preserving property of $E$ was used in the second equality.
    By the faithfulness of $\tau$, we have $ab=ba$ for any $b\in B$.
\end{proof}

Let $P_{n}(X)$ be the element of the algebraic free product $B\langle X\rangle$ of $B$ and $\mathbb{C}\langle X\rangle$ defined by
\[
P_{n}(X)
=e_{1}Xe_{1}+2e_{2}Xe_{2}+\cdots+ne_{n}Xe_{n}.
\]
It is clear that $E[\mathrm{ev}_{a}(P_{n}(X))]=0$ by the assumption of $E[a]=0$. Also, we observe that
\begin{align*}
    \,&\left\|\mathrm{ev}_{a}\left(P_{n}(X)\right)\right\|_{\tau}^2
    =
    \tau\left(
    P_{n}(a)^*P_{n}(a)
    \right)\\
    &=
    \sum_{j,k=1}^{n}\tau\left(je_{j}ae_{j}\cdot ke_{k}ae_{k}\right)
    =
    \sum_{k=1}^{n}k^2
    \tau
    \left(
    e_{k}E\left[a^2\right]e_{k}
    \right)\\
    &=
    \sum_{k=1}^{n}k^2
    \tau
    \left(
    e_{k}\eta(1)e_{k}
    \right)
    =
    \sum_{k=1}^{n}k^2
    \tau
    \left(
    e_{k}
    \right)=\frac{6}{\pi^2}n\to\infty
\end{align*}
as $n\to\infty$, where Lemma \ref{Lem_commu} (i.e., $e_{k}a=ae_{k}$) was used in the third equality, and that
\begin{align*}
    \,&
    \left\|
    \mathrm{ev}_{a}^{\otimes2}
    \left(
    \partial_{X:B}P_{n}(X)
    \right)
    \right\|^2_{\tau\otimes\tau}
    =
    \left\|
    \sum_{k=1}^{n}
    ke_{k}\otimes e_{k}
    \right\|^2_{\tau\otimes\tau}\\
    &=\sum_{j,k=1}^{n}
    (\tau\otimes\tau)
    \left(
    (je_{j}\otimes e_{j})\cdot (ke_{k}\otimes e_{k})
    \right)
    =\sum_{k=1}^{n}k^2
    (\tau\otimes\tau)
    \left(
    e_{k}\otimes e_{k}
    \right)\\
    &=\sum_{k=1}^{n}k^2\cdot
    \left(
    \frac{6}{\pi^2}
    \right)^2\left(\frac{1}{k^2}\right)^2
    =
    \left(
    \frac{6}{\pi^2}
    \right)^2\sum_{k=1}^{n}\frac{1}{k^2}\leq\frac{6}{\pi^2}.
\end{align*}
Therefore, we are arriving at the following:

\begin{Thm}\label{thm_counterexample}
    Let $(A,B,\tau,E)$ be a tracial $B$-valued $C^*$-probability space (resp. $W^*$-probability space) and $a$ be a self-adjont $B$-valued random variable with mean $0$ and variance $\eta:B\to B$ such that $\eta(b)=b$, $b\in B$. Assume that there exist a family $(e_{m})_{m=1}^{\infty}$ of projections in $B$ such that $e_{n}e_{m}=\delta_{n,m}e_{n}$ and $\tau\left(e_{n}\right)=\frac{6}{\pi^2}\frac{1}{n^2}$ for each $n\in\mathbb{N}$. Then, there is no universal constant $C>0$ such that
    \[
    \left\|
    P(a)-E\left[P(a)\right]
    \right\|_{\tau}
    \leq C
    \left\|
    \mathrm{ev}_{a}^{\otimes2}
    \left(
    \partial_{X:B}
    \left[
    P(X)
    \right]
    \right)
    \right\|_{\tau\otimes\tau}
    \]
    holds for any $P(X)\in B\langle X\rangle$.
\end{Thm}

\subsection{Some remarks}
Let $1\in B\subset A$ be an inclusion of von Neumann algebras with faithful normal tracial state $\tau:A\to\mathbb{C}$ and a $\tau$-preserving (unique) conditional expectation $E:A\to B$. Recall the notion of (algebraic version of) conjugate variable due to Voiculescu \cite[section 3]{v98}. An element $a=a^*\in A$ is said to have a \textit{conjugate variable} $\xi\in L^2(A,\tau)$ if we have $\langle1\otimes1,\mathrm{ev}_{a}^{\otimes2}(\partial_{X:B}P)\rangle_{\tau\otimes\tau}=\langle\xi,\mathrm{ev}_{a}(P)\rangle_{\tau}$ for any $P\in B\langle X\rangle$. 

Similarly to \cite[Remark 2.3]{l24}, we can see the if $a=a^*\in A$ has a conjugate variable (in the above sense), then we have $\ker(\mathrm{ev}_{a})\subset \ker(\mathrm{ev}_{a}^{\otimes2}\circ\partial_{X:B})$.


Related to the setting of Theorem \ref{thm_counterexample}, we have the following:
\begin{Prop}
    Assume that $B\not=\mathbb{C}$. If $a=a^*\in A$ commutes with all elements of $B$, then $\ker(\mathrm{ev}_{a})\not\subset\ker(\mathrm{ev}_{a}^{\otimes2}\circ\partial_{X:B})$, and hence $a$ has no conjugate variable.
\end{Prop}
\begin{proof}
    For any $b\in B\setminus\{0\}$, set $P_{b}(X)=bX-Xb$. It is clear that $P_{b}\in\ker(\mathrm{ev}_{a})$ by assumption. On the other hand, we have $\partial_{X:B}P_{b}(X)=b\otimes1-1\otimes b$. In general, $b\otimes1\not=1\otimes b$. (For example, this is the case of $\ker(b)\not=\{0\}$.) Thus, we have $\ker(\mathrm{ev}_{a})\not\subset\ker(\mathrm{ev}_{a}^{\otimes2}\circ\partial_{X:B})$.
\end{proof}

Thus, the following statement should be a precise question concerning Voiculescu's conjecture:
\begin{Que}
    Under the assumption that $a=a^*\in A$ has a conjugate variable, does inequality
    \[
    \|P(a)-E[P(a)]\|_{\tau}^2\leq C\|\mathrm{ev}_{a}^{\otimes2}(\partial_{X:B}P)\|_{\tau\otimes\tau}^{2}
    \]
    hold for any $P\in B\langle X\rangle$ ?
\end{Que}

\section*{Acknowledgements}
The author acknowledges his supervisor, Professor Yoshimichi Ueda, for conversations and editorial support on this paper. The author also would like to thank Dr. Akihiro Miyagawa and Dr. Ryosuke Sato for discussions on divergence operators. Finally, the author would like to express his thanks to Dr. David Jekel for his meaningful comments.

This work was financially supported by JST SPRING, Grant Number JPMJSP2125. The author would like to take this opportunity to thank the
“THERS Make New Standards Program for the Next Generation Researchers”.
}

\end{document}